\documentclass[11pt, reqno]{amsart}

\usepackage[utf8]{inputenc}
\usepackage[letterpaper, margin = 0.9in]{geometry}

\usepackage{graphicx}
\graphicspath{ {Figures/} }
\usepackage{subcaption}
\usepackage[bitstream-charter]{mathdesign}

\usepackage{mathtools, amssymb, bm, amsthm}

\newtheorem{lemma}{Lemma}
\newtheorem{proposition}{Proposition}
\newtheorem{corollary}{Corollary}

\theoremstyle{definition}
\newtheorem{problem}{Problem}
\newtheorem{definition}{Definition}
\newtheorem{remark}{Remark}

\usepackage{array}
\usepackage{enumerate}
\usepackage{enumitem}
\usepackage{cite}
\usepackage{xcolor}
\usepackage{multirow}
\usepackage{booktabs}
\usepackage{amsaddr}

\usepackage{hyperref}
\hypersetup{colorlinks = true, linkcolor = blue, citecolor = blue, urlcolor = blue}

\allowdisplaybreaks

\DeclareMathOperator*{\T}{\mathsf{T}}
\DeclareMathOperator*{\argmin}{arg\,min}
\DeclareMathOperator*{\sign}{sign}
\DeclareMathOperator*{\diag}{diag}
\DeclareMathOperator*{\U}{\mathcal{U}}
\DeclareMathOperator*{\Uc}{\mathcal{U}_c}
\DeclareMathOperator*{\Uuc}{\mathcal{U}_{uc}}

\renewcommand{\qedsymbol}{$\blacksquare$}

\title[Approximate Energetic Resilience of Nonlinear Systems under Partial Loss of Control Authority]{Approximate Energetic Resilience of Nonlinear Systems under\\Partial Loss of Control Authority}
\author{Ram Padmanabhan and Melkior Ornik}
\address{University of Illinois Urbana-Champaign, Urbana, IL 61801, USA}
\thanks{This work was supported by Air Force Office of Scientific Research grant FA9550-23-1-0131 and NASA University Leadership Initiative grant 80NSSC22M0070. (Corresponding Author: Ram Padmanabhan. Email: {\tt ramp3@illinois.edu.})}

\begin{document}

\begin{abstract}
In this paper, we quantify the resilience of nonlinear dynamical systems by studying the increased energy used by all inputs of a system that suffers a partial loss of control authority, either through actuator malfunctions or through adversarial attacks. To quantify the maximal increase in energy, we introduce the notion of an energetic resilience metric. {Prior work in this particular setting does not consider general nonlinear dynamical systems. In developing this framework, we first consider the special case of linear driftless systems and recall the energies in the control signal in the nominal and malfunctioning systems. Using these energies, we derive a bound on the energetic resilience metric. For general nonlinear systems, we first obtain a condition} on the mean value of the control signal in both the nominal and malfunctioning systems, which allows us to approximate the energy in the control. We then obtain a worst-case approximation of this energy for the malfunctioning system, over all malfunctioning inputs. {Assuming this approximation is exact,} we derive bounds on the energetic resilience metric when control authority is lost over one actuator. A set of simulation examples demonstrate that the metric is useful in quantifying the resilience of the system without significant conservatism, despite the approximations used in obtaining control energies for nonlinear systems.
\end{abstract}

\maketitle 

\vspace{-0.5cm}

\section{Introduction} \label{sec:Introduction}
In the process of control system design, ensuring resilience to system faults or adversarial attacks is a key issue. Control systems operate in environments with undesirable or uncontrolled influences, which often present themselves as changes to the structure of actuated inputs to the system. For instance, the disastrous crash of the F-117A Nighthawk aircraft in 1997 resulted from missing fasteners between the elevon hydraulic actuator and the wing \cite{F117A}, contributing to uncontrolled deflections of the wing. More recently, a software fault affected the docking of the Nauka module to the International Space Station (ISS) in 2021. An unexpected firing of the module's thrusters led to the tilting of the ISS by up to $540^{\circ}$ for over $45$ minutes. This was described as a ``loss of attitude control'' by the National Aeronautics and Space Administration (NASA) \cite{Nauka}. To counteract this uncontrolled misfiring, other thrusters on the ISS were fired to return the ISS to its nominal position. While this issue was successfully mitigated, {the mitigating action could have resulted in increased resource consumption, such as time spent and fuel consumed. Another example of an undesirable influence was discovered in a 2024 study by Lemon \cite{AL1, AL2} in the context of traffic light control systems suffering from potential adversarial attacks.} Due to a lack of authentication software, it was found that malicious actors could gain control over the light timings in a network of traffic lights. Light timings are among the primary inputs in a large-scale traffic control system, and manipulating them can clearly lead to significant, potentially disastrous consequences. 

Changes to the input structure of a control system, and in particular actuator faults, also account for $44\%$ of all faults in spacecraft attitude control systems \cite{RS03, HLXZ21}. Building redundancy in actuators \cite{SP90, Grossman95} or using control reconfiguration schemes \cite{Grossman95, BFP19, HLXZ21} are common strategies to address this issue, but can lead to significant costs in control system design and operation. In particular, redundant actuators may suffer from failure of individual components \cite{Steffen08} or unintended effects such as counteracting other actuators \cite{Liu12}. Notably, these faults can prevent a control system from achieving a performance objective such as reachability, safety or optimality with respect to a performance criterion. Moreover, mitigating such issues with classical fault-tolerant, robust or adaptive control strategies is not guaranteed to succeed. Fault-tolerant control only addresses actuator failures limited to reduced effectiveness while still maintaining controllability \cite{XHS13, AH19}, or actuators ``locked'' into producing constant inputs \cite{TCJ02}. Considering uncontrolled inputs as external disturbances within the paradigm of robust control theory may fail as shown in \cite{BO22b}, as robust control requires disturbances to have much smaller magnitudes than the controlled inputs \cite{WZ01}, which is not guaranteed. Fast changes in uncontrolled inputs within a short period of time also lead to the failure of adaptive control techniques \cite{WZ01, AD08}.

Motivated by the difficulties in  addressing such failure modes, in this paper we consider nonlinear dynamical systems that suffer from a \emph{loss of control authority} over a subset of their actuators. Such a failure separates the inputs into \emph{controlled} and \emph{uncontrolled} sets. Uncontrolled inputs can potentially be chosen by an adversary and take on any values in their admissible set. However, these values are measurable, and controlled inputs, which are under the authority of the system, can use these measurements to compensate for any undesirable effects from the uncontrolled inputs. In a general nonlinear system, addressing the problem of designing controlled inputs in this setting is difficult. Rather than designing these controlled inputs, our objective is to quantify the additional energy in the control inputs used to compensate for the effect of uncontrolled inputs, compared to the \emph{nominal} case where control authority is not lost. We propose the metric of \emph{energetic resilience} to quantify this additional energy. We remark that the maximal additional energy is directly related to consumption of resources such as fuel in practical vehicular systems that may lose authority over some inputs. In such settings, fuel economy is an important consideration \cite{Elias22, MS16}. Thus, such a quantification of the increased resource consumption can enable an understanding of the maximal capacity required.

Our line of work contrasts with earlier approaches to quantify resilience in systems and networks. Bouvier and Ornik \cite{BO20} derived conditions for reachability of a target set under any uncontrolled inputs for linear systems, termed \emph{resilient reachability}. In a subsequent extension \cite{BO22b}, a control law that enables a linear system to achieve a target despite an actuator malfunction was designed using the above conditions. In \cite{BXO21}, a \emph{quantitative resilience} metric was introduced by Bouvier \emph{et al.} for linear driftless systems. This metric was defined as the maximal ratio of minimal reach times to achieve a target state between the nominal and malfunctioning systems. Such a metric quantifies the maximal additional time taken by the malfunctioning system, compared to a nominal system, in achieving its objective. This metric was further used to quantify resilience in generalized integrators in \cite{BXO23}, general linear systems in \cite{BO22d, BO23} and cyber-physical networks in \cite{BNO24}. 

In contrast to the above works, the energetic resilience metric we use considers minimal control energies to achieve a target rather than minimal reach times as with quantitative resilience. This metric was first introduced by Padmanabhan \emph{et al.} \cite{PBDO24} to quantify the maximal \emph{cost of disturbance} for a linear system affected by bounded disturbances. {We note that a metric for energetic resilience was proposed in \cite{PO25a} for linear driftless systems. However, that metric is not meaningful for general nonlinear systems, motivating the approach in this paper. We also remark that none of the above works consider a loss of control authority for general nonlinear dynamical systems, and this paper presents our first steps in answering this problem.}

{Recent advances have been made in addressing the general challenge of resilience for nonlinear systems. De Persis and Tesi \cite{PT14} and Zhang \emph{et al.} \cite{ZRD23} consider Denial-of-Service attacks on control and cyber-physical systems with nonlinear dynamics. Along similar lines, Pang \emph{et al.} \cite{PXG20} discuss resilient control for nonlinear cyber-physical systems subjected to network attacks. Zhao \emph{et al.} \cite{ZZ24} and Jia \emph{et al.} \cite{JFXL25} also address the issues of deception and Denial-of-Service attacks over cyber-physical networks. In the problem of resilience for nonlinear systems subjected to external disturbances, Saoud \emph{et al.} \cite{SJS24} use temporal logic specifications to introduce and quantify a metric for resilience. In a similar setting, Zhou \emph{et al.} \cite{ZWW24} use robust passivity theory to design resilient controllers for nonlinear systems with bounded disturbances.} {However, to the best of our knowledge, none of these recent advances address actuator malfunctions or adversaries gaining control over a set of inputs. Our objective is to address resilience in such nonlinear dynamical systems with a set of malfunctioning inputs.}

{The contributions of this paper are organized as follows.
\begin{itemize}[leftmargin = 15pt, topsep = 0pt]
	\item In Section \ref{sec:Formulation}, we describe the structure of the nonlinear system we analyze, as well as providing the definition of key quantities derived throughout this paper. In particular, we define a resilience metric which aims to quantify the additional energy required to compensate the effect of malfunctioning inputs. 
	\item In Section \ref{sec:Res_Dft}, we consider the special case of linear driftless systems and recall some relevant results from \cite{PO25a}. Using these results, we briefly discuss a problem of designing the optimal final time to minimize the energy required to achieve a task under a malfunction. 
	\item Further, in Section \ref{sec:Res_Dft} we consider the case when control authority is lost over one actuator, deriving the maximal total energy used by the linear driftless system in the malfunctioning case. This is used to derive a bound on the resilience metric.
	\item From Section \ref{sec:Energies}, we address general nonlinear dynamical systems. We first obtain a condition on the mean value of the control signal that achieves a specific task in both the nominal and malfunctioning cases. This condition is used to approximate the minimum control energies to achieve the task in both cases, obtaining an approximate upper bound in the malfunctioning case.
	\item In Section \ref{sec:Energies}, we also analyze the special case when authority is lost over one actuator, obtaining the approximate malfunctioning energy in this particular setting.
	\item In Section \ref{sec:Resilience}, under this special case, we obtain a bound on the resilience metric which subsequently bounds from above the maximal additional energy required to achieve a task in the malfunctioning case, compared to the nominal case. 
	\item We present a set of simulation examples on linear driftless, linear and nonlinear dynamical systems in Section \ref{sec:Examples}. Here, we show that the bounds on the resilience metric derived in Sections \ref{sec:Res_Dft} and \ref{sec:Resilience} are useful in characterizing the worst-case control energy required under a variety of malfunctioning inputs. In particular, this bound is not significantly conservative for nonlinear systems despite the approximations used in Section \ref{sec:Energies} to obtain nominal and malfunctioning energies.
\end{itemize}
}

\subsection{Notation and Facts}
The set $\mathbb{R}^+ \coloneqq [0, \infty)$ is the set of all non-negative real numbers. For scalars $z \in \mathbb{R}$, define the $\sign(\cdot)$ function as $\sign(z) = z/|z| \in \{-1, +1\}$ if $z \neq 0$, with $\sign(0) = 0$. For vectors $z \in \mathbb{R}^n$, the $\sign(\cdot)$ function operates elementwise on $z$. The $p$-norm of a vector $x \in \mathbb{R}^n$ is defined as $\|x\|_p \coloneqq \left(\sum_{i=1}^{n}|x_i|^p\right)^{1/p}$, with $\|x\|_{\infty} \coloneqq \max_i |x_i|$. For a matrix $L \in \mathbb{R}^{p\times q}$ with entries indexed $l_{ij}$, define the induced matrix norms $\|L\|_1 \coloneqq \max_j \sum_{i=1}^{p} |l_{ij}|$ and $\|L\|_{\infty} \coloneqq \max_i \sum_{j=1}^{q} |l_{ij}|$. The Moore-Penrose inverse \cite{P55}, also called the pseudoinverse of $L$, is denoted $L^{\dagger}$. For a piecewise continuous function $u:[0, t_f]\to \mathbb{R}^p$, its $\mathcal{L}_2$ norm is defined as $\|u\|_{\mathcal{L}_2} \coloneqq \sqrt{\int_{t=0}^{t_f}\|u(t)\|_{2}^{2}~\mathrm{d}t}$, and represents the \emph{energy} in the function over the interval $[0, t_f]$. The mean value of such a function is defined as $\overline{u} \coloneqq \frac{1}{t_f} \int_{0}^{t_f} u(t)\mathrm{d}t$.

For two vectors or matrices $M$ and $N$ such that their product $MN$ can be defined, the sub-multiplicative property of norms \cite{Horn} states that $\|MN\| \leq \|M\|\|N\|$, where $\|\cdot\|$ is any vector or induced matrix norm. For any two vectors $x \in \mathbb{R}^n$ and $y \in \mathbb{R}^n$, the Cauchy-Schwarz inequality \cite{Horn} can be written as $|x^Ty| \leq \|x\|_2\|y\|_2$. The minimum and maximum eigenvalues of a symmetric matrix $P \in \mathbb{R}^{n\times n}$ are denoted $\lambda_{\min}(P)$ and $\lambda_{\max}(P)$. For such a matrix, the Rayleigh inequality \cite{Horn} $\lambda_{\min}(P)\|x\|_{2}^{2} \leq x^TPx \leq \lambda_{\max}(P)\|x\|_{2}^{2}$ holds for any vector $x \in \mathbb{R}^n$. For a convex function $\varphi:\mathbb{R}^p\to\mathbb{R}$ and piecewise continuous function $u:[0, t_f]\to \mathbb{R}^p$, Jensen's inequality \cite[Chapter 1]{NPBook} states that $\varphi\left(\frac{1}{t_f}\int_{0}^{t_f}u(t)\mathrm{d}t\right) \leq \frac{1}{t_f}\int_{0}^{t_f}\varphi(u(t))\mathrm{d}t$. Given a differential inequality $\dot{q}(t) \leq a(t)q(t)$ where $q$ and $a$ are continuous functions on $t \in [0, t_f]$, Gr\"onwall's inequality \cite[Theorem 1.2.1]{BGP} states that $q(t) \leq q(0)\mathrm{exp}\left(\int_{0}^{t}a(t)\mathrm{d}t\right)$ for all $t \in [0, t_f]$.

\section{Problem Formulation} \label{sec:Formulation}
In this article, we consider continuous-time, time-invariant nonlinear dynamical systems of the form
\begin{equation} \label{eq:Nominal}
	\dot{x}(t) = f(x(t)) + g(x(t))u(t), ~x(0) = x_0 \neq 0,
\end{equation}
where the state $x(t) \in \mathbb{R}^n$ and the control $u(t) \in \mathbb{R}^{m+p}$ for all $t$, $f:\mathbb{R}^n \to \mathbb{R}^n$ and $g: \mathbb{R}^n \to \mathbb{R}^{n\times (m+p)}$. We assume the functions $f$ and $g$ satisfy the following conditions:
\begin{alignat}{2}
f(x_1) = f(x_2) &+ d_f(x_1, x_2) &&\text{ where } \|d_f(x_1, x_2)\|_{\infty} \leq D_f \|x_1-x_2\|_{\infty}, \label{eq:f_cond} \\
\text{ and } g(x_1) = g(x_2) &+ d_g(x_1, x_2) &&\text{ where } \|d_g(x_1, x_2)\|_{\infty} \leq D_g \|x_1 - x_2\|_{\infty}. \label{eq:g_cond}
\end{alignat}
In other words, $f$ and $g$ are $D_f$ and $D_g$-Lipschitz continuous in the $\infty$-norm respectively. Further, the control $u$ lies in the admissible set $\U$, where
\begin{equation} \label{eq:SetU}
\U \coloneqq \left\{u: \mathbb{R}^+ \to \mathbb{R}^{m+p}: u \text{ is piecewise continuous in $t$, } \|u(t)\|_{\infty} \leq 1 \text{ for all $t$}\right\}
\end{equation}
in line with previous works \cite{BO23, PO25a}. {The condition $\|u(t)\|_{\infty} \leq 1$ represents actuation limits in any practical control system.} Conditions \eqref{eq:f_cond} and \eqref{eq:g_cond}, and the admissible set $\U$ in \eqref{eq:SetU} guarantee that a unique solution $x(t)$ to the system \eqref{eq:Nominal} exists \cite{Khalil}.

Consider a malfunction that causes the system \eqref{eq:Nominal} to lose authority over $p$ of its $m+p$ actuators. Then, the control $u$ and the function $g$ can be split into controlled and uncontrolled components as
\begin{equation} \label{eq:Malfunctioning}
	\dot{x}(t) = f(x(t)) + g_{c}(x(t))u_c(t) + g_{uc}(x(t))u_{uc}(t),
\end{equation}
where the subscripts $c$ and $uc$ denote \emph{controlled} and \emph{uncontrolled} respectively. Here, the functions $g_c:\mathbb{R}^n \to \mathbb{R}^{n\times m}$ and $g_{uc}:\mathbb{R}^n \to \mathbb{R}^{n\times p}$ are such that $g(\cdot) = \left[g_c(\cdot), ~g_{uc}(\cdot)\right]$ satisfies \eqref{eq:g_cond}. Based on the set of admissible controls $u \in \U$ with $u(t) = \left[u_{c}^{\T}(t),~ u_{uc}^{\T}(t)\right]^{\T}$, the controlled and uncontrolled inputs satisfy $u_c \in \Uc$ and $u_{uc} \in \Uuc$, where
\begin{align}
\Uc &\coloneqq \left\{u_c: \mathbb{R}^+ \to \mathbb{R}^{m}: u_c \text{ is piecewise continuous in $t$, } \|u_c(t)\|_{\infty} \leq 1 ~\text{for all $t$}\right\}, \label{eq:SetUc} \\
\Uuc &\coloneqq \left\{u_{uc}: \mathbb{R}^+ \to \mathbb{R}^{p}: u_{uc} \text{ is piecewise continuous in $t$, } \|u_{uc}(t)\|_{\infty} \leq 1 ~\text{for all $t$}\right\}. \label{eq:SetUuc}
\end{align}

In this setting, we consider the task of \emph{fixed-time reachability}, i.e., achieving a specified target state $x(t_f) = x_{tg} \in \mathbb{R}^n$ from an initial state $x(0) = x_0 \in \mathbb{R}^n$ at a specified final time $t_f$. Under the nominal dynamics \eqref{eq:Nominal}, the system uses the input $u \in \U$, while under the malfunctioning dynamics \eqref{eq:Malfunctioning}, the system must use only the controlled input $u_c \in \Uc$ to achieve this task. We note that in the latter case, the uncontrolled input $u_{uc}$ is arbitrarily chosen --- potentially by an adversary --- in the space $\Uuc$. However, we assume this choice of $u_{uc}$ is observable, potentially \emph{a priori}, and hence its effect can be compensated for in the design of $u_c$. Achieving this task despite the effect of $u_{uc}$ corresponds to the notion of \emph{resilient reachability}, formally defined below.

\begin{definition}
A target $x_{tg} \in \mathbb{R}^n$ is \emph{resiliently reachable} at a given final time $t_f$ from $x_0 \in \mathbb{R}^n$ if for all uncontrolled inputs $u_{uc} \in \Uuc$ there exists a controlled input $u_c \in \Uc$, possibly dependent on $u_{uc}$, such that $x(t_f) = x_{tg}$.
\end{definition}

Then, our first objective can be stated as follows.

\begin{problem} \label{prob:char}
Characterize the control input $u \in \U$ in \eqref{eq:Nominal} (resp. $u_c \in \Uc$ in \eqref{eq:Malfunctioning}) that ensures a target $x_{tg} \in \mathbb{R}^n$ is reachable (resp. resiliently reachable).
\end{problem}

In Section \ref{sec:Energies}, we use the solution of the nominal \eqref{eq:Nominal} and malfunctioning \eqref{eq:Malfunctioning} systems to obtain a condition on the controls $u \in \U$ and $u_c \in \Uc$ that achieve fixed-time reachability. Further, in both the nominal and malfunctioning cases, we are interested in control signals that use minimum energy to achieve $x(t_f) = x_{tg}$. In such a context, we make the following definitions, also used in \cite{PO25a}.

\begin{definition}[Nominal Energy]
The \emph{nominal energy} is the minimum energy in the control input $u \in \U$ such that a target $x_{tg}$ is reached from $x_0$ in a given time $t_f$ in the system \eqref{eq:Nominal}:
\begin{equation} \label{eq:EN_def}
E_{N}^{*}(x_0, x_{tg}, t_f) \coloneqq \inf_{u \in \U} \left\{\|u\|_{\mathcal{L}_2}^{2} \text{ s.t. } x(t_f) = x_{tg} \text{ in } \eqref{eq:Nominal}\right\}.
\end{equation}
\end{definition}

\begin{definition}[Malfunctioning Energy]
The \emph{malfunctioning energy} is the minimum energy in the controlled input $u_c \in \Uc$ such that a target $x_{tg}$ is reached from $x_0$ in a given time $t_f$ and for a given uncontrolled input $u_{uc} \in \Uuc$, in the system \eqref{eq:Malfunctioning}:
\begin{equation} \label{eq:EM_def}
E_{M}^{*}(x_0,x_{tg},t_f,u_{uc}) \coloneqq \inf_{u_c(u_{uc}) \in \Uc} \left\{\|u_{c}(u_{uc})\|_{\mathcal{L}_2}^{2} \text{ s.t. } x(t_f) = x_{tg} \text{ in } \eqref{eq:Malfunctioning} \text{ for the given $u_{uc}$}\right\},
\end{equation}
where the dependence of $u_c$ on $u_{uc}$ is explicitly denoted with a slight abuse of notation.
\end{definition}

In \eqref{eq:EM_def}, the controlled input $u_c$ may depend not only on past and current values of $u_{uc}$, but potentially also future values of this uncontrolled input. The dependence on future values considers any prior knowledge we may have about uncontrolled inputs, such as the corresponding actuators locking in place or acting with reduced effectiveness \cite{AH19, TCJ02}. In both \eqref{eq:EN_def} and \eqref{eq:EM_def}, deriving the corresponding optimal control signals is not tractable \cite{Khalil} due to the nonlinear nature of the systems \eqref{eq:Nominal} and \eqref{eq:Malfunctioning}. However, our objective in this paper is not to design these optimal control signals, but rather to quantify how much larger the malfunctioning energy may be, compared to the nominal energy. 

The controlled input $u_c$ seeks to compensate for the effect of $u_{uc}$ and ensure that $x_{tg}$ is reachable from $x_0$ in time $t_f$, while also using minimal energy \eqref{eq:EM_def}. On the other hand, the uncontrolled input $u_{uc}$ also contributes to the energy used by the actuators of the system. In a system \eqref{eq:Malfunctioning} that has lost authority over some of its actuators, it is important to characterize the maximal additional energy that may be used by all actuators. In practical cases, this maximal additional energy is directly related to consumption of resources such as fuel in vehicular systems. Quantifying such an increase in resource consumption is essential to understand the maximum capacity of resources required. In an adversarial setting, the uncontrolled input $u_{uc}$ may be chosen to maximize the total energy used by the system. We thus use the following definitions, first introduced in \cite{PO25a}, that consider the total energy used by all actuators of the system \eqref{eq:Malfunctioning}.

\begin{definition}[Total Energy]
The \emph{total energy} is the energy used by all actuators of the system \eqref{eq:Malfunctioning} for a given uncontrolled input $u_{uc}$, assuming the optimal controlled input $u_c(u_{uc})$ in \eqref{eq:EM_def} is used:
\begin{equation} \label{eq:ET_def}
E_{T}(x_0, x_{tg}, t_f, u_{uc}) \coloneqq E_{M}^{*}(x_0, x_{tg}, t_f, u_{uc}) + \|u_{uc}\|_{\mathcal{L}_2}^{2}. 
\end{equation}
\end{definition}

\begin{definition}[Worst-case Total Energy]
The \emph{worst-case total energy} is the maximal value of the total energy \eqref{eq:ET_def} over all possible uncontrolled inputs $u_{uc}$:
\begin{equation} \label{eq:ET_bar_def}
\overline{E}_T(x_0, x_{tg}, t_f) \coloneqq \sup_{u_{uc} \in \Uuc} E_{T}(x_0, x_{tg}, t_f, u_{uc}) = \sup_{u_{uc} \in \Uuc} \left\{E_{M}^{*}(x_0, x_{tg}, t_f, u_{uc}) + \|u_{uc}\|_{\mathcal{L}_2}^{2}\right\}.
\end{equation}
The worst-case total energy quantifies the maximal effect of $u_{uc}$ on the total energy used by all actuators of the system \eqref{eq:Malfunctioning}.
\end{definition}

Finally, note that the worst-case total energy \eqref{eq:ET_bar_def} can be considerably larger than the nominal energy \eqref{eq:EN_def} due to the action of the uncontrolled input $u_{uc}$. Quantifying this increased energy is an important problem to understand additional consumption of resources in practical settings. We thus arrive at our second objective{, that of quantifying this additional energy for general nonlinear systems.}

\begin{problem} \label{prob:resilience}
Quantify the maximal increase in energy used by the malfunctioning system \eqref{eq:Malfunctioning}, in comparison to the nominal system \eqref{eq:Nominal}.
\end{problem}

To quantify such an increase in energy, we define the following \emph{energetic resilience} metric, similar to that defined in \cite{PBDO24}.

\begin{definition}[Energetic Resilience]
The \emph{energetic resilience} of system \eqref{eq:Malfunctioning} is defined for an initial condition $x_0$ no farther than a distance of $R$ from the target state $x_{tg}$, as
\begin{equation} \label{eq:rA_def}
r_A(t_f, R) \coloneqq \sup_{\|x_0-x_{tg}\|_2 \leq R} \overline{E}_T(x_0,x_{tg},t_f) - E_{N}^{*}(x_0, x_{tg}, t_f).
\end{equation}
\end{definition}

The significance of this metric is evident. Based on the definition \eqref{eq:rA_def}, the malfunctioning system \eqref{eq:Malfunctioning} will use \emph{at most} $r_A(t_f,R)$ more energy than the nominal system \eqref{eq:Nominal} to achieve fixed-time reachability, for a given final time $t_f$ and $x_0$ no farther than $R$ from $x_{tg}$. The constraint $\|x_0-x_{tg}\|_2 \leq R$ is required to ensure that the value of $r_A(t_f, R)$ is finite.

{We note that an alternative metric for energetic resilience was considered in \cite{PO25a}. However, the treatment in \cite{PO25a} considered only linear driftless systems, and the corresponding metric is not meaningful for general nonlinear systems due to differences in system responses as discussed in Section \ref{sec:Energies}. In contrast, the metric \eqref{eq:rA_def} we consider in this paper can be extended to nonlinear systems. In the next section, we also begin by considering linear driftless systems and deriving a bound on the metric \eqref{eq:rA_def}. The results in this section are used as a basis for addressing resilience of general nonlinear systems in Section \ref{sec:Energies} and beyond.}

\section{Resilience of Linear Driftless Systems} \label{sec:Res_Dft}
{In this section, we consider the special case of linear driftless systems, i.e., $f(x(t)) \equiv 0$ and $g(x(t)) \equiv B$ for all $x(t)$ in \eqref{eq:Nominal}. We first  review the expressions for the nominal \eqref{eq:EN_def}, malfunctioning \eqref{eq:EM_def} and worst-case total \eqref{eq:ET_bar_def} energies for linear driftless systems derived in \cite{PO25a}. Based on these, we briefly consider the problem of designing the optimal final time to minimize malfunctioning energy. We then address the central problem of bounding the metric \eqref{eq:rA_def}, in the case of losing authority over one actuator.}
In the case of linear driftless systems, we have $f(x(t)) \equiv 0$ and $g(x(t)) \equiv B$ so that $D_f = D_g = 0$ and the system \eqref{eq:Nominal} reduces to
\begin{equation} \label{eq:Nominal_Dft}
	\dot{x}(t) = Bu(t), ~x(0) = x_0.
\end{equation}
Thus, under a malfunction of the form described in Section \ref{sec:Formulation}, the malfunctioning system \eqref{eq:Malfunctioning} is written as
\begin{equation} \label{eq:Malfunctioning_Dft}
	\dot{x}(t) = B_{c}u_{c}(t) + B_{uc}u_{uc}(t).
\end{equation}
Under the assumption that the nominal dynamics \eqref{eq:Nominal_Dft} are controllable, it was shown that a control input $u^*$ achieves the task if and only if
\begin{equation} \label{eq:uN}
	u^*(t) = \overline{u}^{LS} = -\frac{1}{t_f}B^{\dagger}\tilde{x} ~\text{ for all }~ t \in [0, t_f],
\end{equation}
where $\overline{u}^{LS}$ denotes the mean value of the control with the minimum energy. Note that this control is constant in the interval $[0, t_f]$. Further, the constraint $u \in \U$, i.e., $\|u(t)\|_{\infty} \leq 1$ is satisfied if and only if
\begin{equation} \label{eq:Nominal_Cond}
	t_f \geq \|B^{\dagger}\tilde{x}\|_{\infty}.
\end{equation}
Finally, the nominal energy was written as
\begin{equation} \label{eq:EN_Dft}
	E_{N}^{*}(x_0, x_{tg}, t_f) = \int_{0}^{t_f} u^{*T}(t)u^*(t)\mathrm{d}t = \frac{1}{t_f} \|B^{\dagger}\tilde{x}\|_{2}^{2}.
\end{equation}
Using similar arguments, the malfunctioning energy \eqref{eq:EM_def} was derived in \cite{PO25a} as
\begin{equation} \label{eq:EM_Dft}
	E_{M}^{*}(x_0, x_{tg}, t_f, u_{uc}) = \frac{1}{t_f} \left\|B_{c}^{\dagger}\left(\tilde{x} + t_fB_{uc}\overline{u}_{uc}\right)\right\|_{2}^{2}
\end{equation}
under conditions similar to \eqref{eq:Nominal_Cond} on the final time $t_f$. From this, the worst-case total energy $\overline{E}_T(x_0, x_{tg}, t_f)$ was shown to be bounded from above as follows \cite[Proposition 1]{PO25a}:
\begin{align}
\overline{E}_T(x_0, x_{tg}, t_f) &\leq \frac{1}{t_f}\left\|B_{c}^{\dagger}x_0\right\|_{2}^{2} + t_f\left(\sum_{i = 1}^{p} \lambda_i \|V\|_{1}^{2} + p\right) + 2\|B_{uc}^{T}B_{c}^{\dagger T}B_{c}^{\dagger}x_0\|_1. \label{eq:ETbar_Dft}
\end{align}
In the particular case of losing control authority over one actuator, i.e., $p = 1$, it was also shown that the following closed-form expression can be obtained for the worst-case total energy:
\begin{align}
\overline{E}_T(x_0, x_{tg}, t_f) &= \frac{1}{t_f}\left\|B_{c}^{\dagger}x_0\right\|_{2}^{2} + t_f\left(\|B_{uc}\|_{2}^{2} + 1\right) + 2\left|B_{uc}^{T}B_{c}^{\dagger T}B_{c}^{\dagger}x_0\right|. \label{eq:ETbar_Dft_1Act}
\end{align}

\subsection{Final Time Design} \label{eq:FinalTime}
In the context of the malfunctioning energy \eqref{eq:EM_Dft}, it is useful to investigate the optimal final time $t_{f}^{*}$ that minimizes this energy for a given uncontrolled input $u_{uc}$. This is useful in designing control signals for tasks with a variable completion time and constraints on the total available control energy. In other words, we wish to solve the following problem:
\begin{align} \label{eq:EM_tf}
t_{f}^{*} &= \argmin_{t_f} E_{M}^{*}(x_0, x_{tg}, t_f, u_{uc}) \nonumber \\
&= \argmin_{t_f} \frac{1}{t_f} \left\|B_{c}^{\dagger}\left(x_0 + t_fB_{uc}\overline{u}_{uc}\right)\right\|_{2}^{2}.
\end{align}
Using the first-order stationary condition $\frac{\partial}{\partial t_f} \left(E_{M}^{*}(x_0, x_{tg}, t_f, u_{uc})\right) = 0$, we have
$$
t_f \frac{\partial}{\partial t_f}\left\|B_{c}^{\dagger}\left(x_0 + t_fB_{uc}\overline{u}_{uc}\right)\right\|_{2}^{2} = \left\|B_{c}^{\dagger}\left(x_0 + t_fB_{uc}\overline{u}_{uc}\right)\right\|_{2}^{2},
$$%
with $t_f \neq 0$. {The above equality simplifies to the following quadratic equation in $t_f$:}
$$
t_{f}^{2}\overline{u}_{uc}^{T}B_{uc}^{T}B_{c}^{\dagger T}B_{c}^{\dagger}B_{uc}\overline{u}_{uc} = x_{0}^{T}B_{c}^{\dagger T}B_{c}^{\dagger}x_0.
$$%
Solving for $t_f > 0$, we have:
\begin{equation} \label{eq:tf_min}
	t_{f}^{*} = \frac{\|B_{c}^{\dagger}x_0\|_2}{\|B_{c}^{\dagger}B_{uc}\overline{u}_{uc}\|_2}.
\end{equation}
It is then straightforward to show that 
\begin{align*}
&\frac{\partial^2}{\partial t_{f}^2} \left(E_{M}^{*}(x_0, x_{tg}, t_f, u_{uc})\right) = \frac{2\|B_{c}^{\dagger}x_0\|_{2}^{2}}{t_{f}^{3}} ~\text{ and thus}~ \nonumber \\
&\frac{\partial^2}{\partial t_{f}^2} \left(E_{M}^{*}(x_0, x_{tg}, t_f, u_{uc})\right) \Bigg|_{t_f = t_{f}^{*}} = \frac{2\|B_{c}^{\dagger}B_{uc}\overline{u}_{uc}\|_{2}^{3}}{\|B_{c}^{\dagger}x_0\|_2} > 0.
\end{align*}
Hence $t_{f}^{*}$ in \eqref{eq:tf_min} is the unique minimum for the problem in \eqref{eq:EM_tf}. The result \eqref{eq:tf_min} thus provides the final time that minimizes the malfunctioning energy for a given uncontrolled input $\overline{u}_{uc}$.

\subsection{Resilience of Linear Driftless Systems to the Loss of One Actuator} \label{sec:1Act_Dft}
We now return to the expressions \eqref{eq:EN_Dft}-\eqref{eq:ETbar_Dft_1Act} and derive a bound on the energetic resilience metric \eqref{eq:rA_def} under loss of control authority over one actuator. We note that a bound on $r_A(t_f, R)$ can be derived when control authority is lost over $p > 1$ actuators using \eqref{eq:ETbar_Dft}. However, this bound will be more conservative as \eqref{eq:ETbar_Dft} is only an upper bound an not an exact expression as in \eqref{eq:ETbar_Dft_1Act}.

\begin{proposition} \label{prop:rA_Dft}
For the case when $p = 1$, the additive energetic resilience metric $r_A(t_f, R)$ is bounded from above as follows:
\begin{align}
r_A(t_f, R) &\leq \frac{R^2}{t_f}\lambda_{\max}\left(B_{c}^{\dagger T}B_{c}^{\dagger} - B^{\dagger T}B^{\dagger}\right) + 2R\|B_{uc}^{T}B_{c}^{\dagger T}B_{c}^{\dagger}\|_2 + t_f\left(\|B_{uc}\|_{2}^{2} + 1\right). \label{eq:rA_bound}
\end{align}
\end{proposition}

\begin{proof}
Using \eqref{eq:EN_Dft}, \eqref{eq:ETbar_Dft_1Act} and \eqref{eq:rA_def},
\begin{align}
r_A(t_f, R) &= \sup_{\|x_0\|_2 \leq R} \Biggl\{\frac{1}{t_f}\left\|B_{c}^{\dagger}x_0\right\|_{2}^{2} + t_f\left(\|B_{uc}\|_{2}^{2} + 1\right) + 2\left|B_{uc}^{T}B_{c}^{\dagger T}B_{c}^{\dagger}x_0\right| - \frac{1}{t_f}\left\|B^{\dagger}x_0\right\|_{2}^{2}\Biggr\} \nonumber \\
&\leq t_f\left(\|B_{uc}\|_{2}^{2} + 1\right) + 2R\|B_{uc}^{T}B_{c}^{\dagger T}B_{c}^{\dagger}\|_2 + \sup_{\|x_0\|_2 \leq R} \frac{1}{t_f}x_{0}^{T}\left(B_{c}^{\dagger T}B_{c}^{\dagger} - B^{\dagger T}B^{\dagger}\right)x_0, \label{eq:rA_1}
\end{align}
where the inequality follows from the Cauchy-Schwarz inequality applied to the term $2\left|B_{uc}^{T}B_{c}^{\dagger T}B_{c}^{\dagger}x_0\right|$. Note that $B_{c}^{\dagger T}B_{c}^{\dagger} - B^{\dagger T}B^{\dagger}$ is symmetric, and hence using the Rayleigh inequality and \eqref{eq:rA_1},
\begin{align*}
r_A(t_f, R) &\leq \frac{R^2}{t_f}\lambda_{\max}\left(B_{c}^{\dagger T}B_{c}^{\dagger} - B^{\dagger T}B^{\dagger}\right) + 2R\|B_{uc}^{T}B_{c}^{\dagger T}B_{c}^{\dagger}\|_2 + t_f\left(\|B_{uc}\|_{2}^{2} + 1\right),
\end{align*}
concluding the proof.
\end{proof}
In the following section, we begin considering the general nonlinear systems \eqref{eq:Nominal} and \eqref{eq:Malfunctioning}, deriving approximations for the energies defined in Section \ref{sec:Formulation}. These approximations rely on the difference in responses between the general nonlinear system \eqref{eq:Nominal} and the linear driftless system \eqref{eq:Nominal_Dft}.

\section{Approximations for Nominal, Malfunctioning and Worst-case Total Energies} \label{sec:Energies}
{We now consider the general nonlinear systems \eqref{eq:Nominal} and \eqref{eq:Malfunctioning}, and derive approximations for the nominal, malfunctioning and worst-case total energies defined in Section \ref{sec:Formulation}. We note that for these nonlinear systems, deriving a closed-form expression for the control that achieves $x(t_f) = x_{tg}$ is generally intractable (see, e.g., \cite{Khalil, LT07}). Similarly, closed-form expressions for the minimum energies \eqref{eq:EN_def} and \eqref{eq:EM_def} and subsequently the total \eqref{eq:ET_def} and worst-case total \eqref{eq:ET_bar_def} energies can generally not be derived \cite{CVOC} in the nonlinear case. To address this, we derive a condition on the mean value of the control in the interval $[0,t_f]$, in both the nominal and malfunctioning cases. This condition is the solution to Problem \ref{prob:char}, and is used to approximate the minimum energies and worst-case total energy. These approximations are used to derive a bound on the energetic resilience metric in Section \ref{sec:Resilience}.}

\subsection{Nominal and Malfunctioning Energies}
First consider the nominal system \eqref{eq:Nominal}. The state trajectory of this system satisfies
\begin{equation} \label{eq:Sol1}
x(t) = x_0 + \int_{0}^{t} f(x(\tau))\mathrm{d}\tau + \int_{0}^{t} g(x(\tau))u(\tau)\mathrm{d}\tau.
\end{equation}
Let $\tilde{x} \coloneqq x_0 - x_{tg}$, $B = g(x_0)$ and $f_0 = f(x_0)$. We assume $B$ has full row rank, similar to the assumption made in Section~\ref{sec:Res_Dft}. If $x(t_f) = x_{tg}$,
\begin{equation} \label{eq:Sol2}
\tilde{x} + t_fB\overline{u} = -\int_{0}^{t_f} f(x(t))\mathrm{d}t - \int_{0}^{t_f} d_g(x(t),x_0)u(t)\mathrm{d}t,
\end{equation}
where we use the property \eqref{eq:g_cond} for $g$, and $\overline{u} = \frac{1}{t_f}\int_{0}^{t_f}u(t)\mathrm{d}t$ denotes the mean value of $u$ in the interval $[0, t_f]$. {Let $v \equiv v(t_f, u) \coloneqq -\int_{0}^{t_f} f(x(t))\mathrm{d}t - \int_{0}^{t_f} d_g(x(t),x_0)u(t)\mathrm{d}t$. If $v$ was a fixed quantity with respect to $u$ and hence $\overline{u}$, \eqref{eq:Sol2} is linear in $\overline{u}$ and we can write the following least-norm solution \cite{P56}:
\begin{equation} \label{eq:uLS}
\overline{u}^{LS} = \frac{1}{t_f} B^{\dagger}(v-\tilde{x}),
\end{equation}
where $B^{\dagger}$ is the pseudoinverse of $B$, defined as in \cite{P55}. Note that such a least-norm solution is guaranteed to satisfy \eqref{eq:Sol2} only if $v-\tilde{x}$ lies in the range space of $B$, and we assume this is true. Next, the solution \eqref{eq:uLS} assumes $v$ is a fixed quantity with respect to $\overline{u}$, which is clearly not true as $v$ depends directly on $u$. Nevertheless, we use the norm of $\overline{u}^{LS}$ as defined in \eqref{eq:uLS} to approximate the nominal energy later in this section. A key step towards such an approximation is to show that $v$ is bounded, which is proved in the following proposition.}

\begin{proposition} \label{prop:v}
The quantity $v = -\int_{0}^{t_f} f(x(t))\mathrm{d}t - \int_{0}^{t_f} d_g(x(t),x_0)u(t)\mathrm{d}t$ is bounded as
\begin{equation} \label{eq:v_bound}
\|v\|_{\infty} \leq \overline{v} \coloneqq \frac{1}{D_f+D_g} \left[c\left(e^{t_f(D_f + D_g)}-1\right)-t_f(D_f+D_g)\|B\|_{\infty}\right]
\end{equation}
where $c = \|f_0\|_{\infty} + \|B\|_{\infty}$.
\end{proposition}
\begin{proof}
See Appendix \ref{app:A}.
\end{proof}

{It is clear that $v(t_f, u)$ denotes the difference between system responses under the driftless dynamics \eqref{eq:Nominal_Dft} and the general nonlinear dynamics \eqref{eq:Nominal} for a given control input $u \in \U$. Proposition \ref{prop:v} thus shows that this difference between system responses is bounded and quantified by $\overline{v}$, where $\overline{v}$ is a function of the system dynamics, initial condition and final time. In particular, when $f(x(t)) \equiv 0$ and $g(x(t)) \equiv B$, the dynamics in \eqref{eq:Nominal} reduce to \eqref{eq:Nominal_Dft} and we obtain $v(t_f, u) = 0$, $\overline{v} = 0$ since $D_f = D_g = 0$ and $f_0 = 0$. Next, let $D_S \coloneqq D_f + D_g$. In investigating how the difference in system responses between \eqref{eq:Nominal} and \eqref{eq:Nominal_Dft} changes as $D_f$ and $D_g$ increase, we note that
$$
\frac{\partial\overline{v}}{\partial D_S} = \frac{c[(t_fD_S-1)e^{t_fD_S} + 1]}{D_{S}^{2}} \geq 0,
$$%
i.e., $\overline{v}$ is a non-decreasing function of $D_S$. Thus, the difference in system responses increases as $D_S$ increases. The facts discussed above are used to motivate the approximations for different energies derived later in this section.}

{We now use the bound $\overline{v}$ to analyze conditions under which the constraint $u \in \U$ is satisfied. Note that $u \in \U$ implies $\|\overline{u}\|_{\infty} \leq 1$. For a fixed $v$, $\overline{u}^{LS}$ as defined in \eqref{eq:uLS} provides an expression for $\overline{u}$ and we can check the condition $\|\overline{u}^{LS}\|_{\infty} \leq 1$. To consider all possible $v$, we use the bound $\|v\|_{\infty} \leq \overline{v}$ and check the following condition for a given $t_f$:}
\begin{equation} \label{eq:ucond}
\max_{\|v\|_{\infty} \leq \overline{v}} \frac{1}{t_f} \left\|B^{\dagger}(v-\tilde{x})\right\|_{\infty} \leq 1.
\end{equation}
{If \eqref{eq:ucond} is satisfied, then $\|\overline{u}^{LS}\|_{\infty} \leq 1$ is true for all $v(t_f, u)$, which is a necessary condition to ensure $u \in \U$ for a given $t_f$.} Furthermore, \eqref{eq:ucond} is equivalent to checking 
\begin{align*}
&\max_{\|v\|_{\infty} \leq \overline{v}} \frac{1}{t_f} \left[B^{\dagger}(v-\tilde{x})\right]_{i} \leq 1 \\
\text{and } &\max_{\|v\|_{\infty} \leq \overline{v}} -\frac{1}{t_f} \left[B^{\dagger}(v-\tilde{x})\right]_{i} \leq 1
\end{align*}
for each $i = 1, \ldots, m+p$, where $[\cdot]_i$ denotes the $i$-th component of the argument. Each of the above conditions is linear in $v$, and thus can be checked efficiently on the vertices of the hypercube $\|v\|_{\infty} \leq \overline{v}$ \cite{LinOpt}.

{We now use \eqref{eq:uLS} to obtain the nominal energy \eqref{eq:EN_def}. For a general nonlinear system of the form \eqref{eq:Nominal}, obtaining closed-form expressions for optimal control signals is generally intractable \cite{CVOC}. However, the norm of $\overline{u}^{LS}$ for a given $v$ can be used to approximate the nominal energy using the following lemma.}

\begin{lemma} \label{lem:approx}
For a piecewise continuous function $z:[0, t_f] \to \mathbb{R}^{n_z}$ with a given mean value $\overline{z} \in \mathbb{R}^{n_z}$, the energy of $z$ is related to its mean value by
\begin{equation} \label{eq:approx}
\int_{0}^{t_f} z^T(t)z(t)\mathrm{d}t \geq t_f \|\overline{z}\|_{2}^{2},
\end{equation}
where equality holds when $z(t) = \overline{z}$ for all $t \in [0, t_f]$.
\end{lemma}
\begin{proof}
The squared norm $\|\cdot\|_{2}^{2}$ is a convex function, Using Jensen's inequality \cite{NPBook}, we have $\|\overline{z}\|_{2}^{2} = \left\|\frac{1}{t_f}\int_{0}^{t_f}z(t)\mathrm{d}t\right\|_{2}^{2} \leq \frac{1}{t_f}\int_{0}^{t_f} \|z(t)\|_{2}^{2}\mathrm{d}t$. From this, the result follows.
\end{proof}

{Any control signal $u$ that achieves fixed-time reachability must satisfy \eqref{eq:Sol2}, where $\overline{u}$ is its mean value. On the basis of Lemma \ref{lem:approx}, we use the right-hand side of \eqref{eq:approx} to approximate the energy in $u$. Further, based on \cite[Lemma 1]{PO25a}, we are interested in the mean value with minimum-norm that solves \eqref{eq:Sol2}. This quantity is given by $\overline{u}^{LS}$ as defined in \eqref{eq:uLS} for a particular $v$.} We now make the central approximation in this article, namely, that 
\begin{equation} \label{eq:ENapprox}
E_{N}^{*}(x_0, x_{tg}, t_f) \approx t_f \|\overline{u}^{LS}\|_{2}^{2},
\end{equation}
thus approximating \eqref{eq:approx} with equality using the minimum-norm mean $\overline{u}^{LS}$. {We remark that the conservatism of this approximation can be measured by comparing it to the energy of any control signal $u \in \U$ that achieves the objective $x(t_f) = x_{tg}$.} 

{Note that the approximation \eqref{eq:ENapprox} is exact when the optimal control is constant, i.e., $u^*(t) = \overline{u}^{LS}$ for all $t \in [0, t_f]$. In particular, this approximation is exact in the case of linear driftless systems where $D_f = D_g = 0$ and $f_0 = 0$, as shown in \cite{PO25a} where the optimal control $u^*(t) = \overline{u}_{LS}$ for all $t \in [0, t_f]$. In this case, the quantities $v(t_f, u)$ and its bound $\overline{v}$ reduce to zero. For general nonlinear systems, this approximation is more accurate for smaller values of $D_f$ and $D_g$, as such settings result in small values of $\overline{v}$. Here, we recall that $\overline{v}$ quantifies the difference in system responses between \eqref{eq:Nominal} and \eqref{eq:Nominal_Dft}, and $\overline{v}$ is a non-decreasing function of $D_S$.} We also use an approximation similar to \eqref{eq:ENapprox} later in the malfunctioning case. These approximations are used to circumvent the difficulties in obtaining closed-form expressions for optimal control signals and the corresponding minimum energies \cite{CVOC}. Under the assumption that this approximation is exact, we obtain the nominal energy
\begin{equation} \label{eq:EN}
E_{N}^{*}(x_0, x_{tg}, t_f) = \frac{1}{t_f} \left\|B^{\dagger}(v-\tilde{x})\right\|_{2}^{2},
\end{equation}
where $\|v\|_{\infty} \leq \overline{v}$ \eqref{eq:v_bound}. Equation \eqref{eq:EN} thus provides the approximation to the nominal energy using the norm of the mean value $\overline{u}^{LS}$ from \eqref{eq:uLS}. 

We now consider the malfunctioning case, which proceeds along similar lines to the nominal case. The state trajectory of the malfunctioning system \eqref{eq:Malfunctioning} is given by
\begin{equation} \label{eq:MalfSol1}
x(t) = x_0 + \int_{0}^{t} f(x(\tau))\mathrm{d}\tau + \int_{0}^{t}g_c(x(\tau))u_c(\tau)\mathrm{d}\tau + \int_{0}^{t} g_{uc}(x(\tau))u_{uc}(\tau)\mathrm{d}\tau.
\end{equation}
Let $B_c \coloneqq g_c(x_0)$ and $B_{uc} \coloneqq g_{uc}(x_0)$. Then, under fixed-time reachability with $x(t_f) = x_{tg}$, 
\begin{equation} \label{eq:MalfSol2}
\tilde{x} + t_fB_c\overline{u}_c + t_fB_{uc}\overline{u}_{uc} = -\int_{0}^{t} f(x(\tau))\mathrm{d}\tau - \int_{0}^{t_f}d_g(x(t),x_0)u(t)\mathrm{d}t
\end{equation}
where $u(t) = \left[u_{c}^{\T}(t), u_{uc}^{\T}(t)\right]^{\T}$, $\overline{u}_c \coloneqq \frac{1}{t_f}\int_{0}^{t_f}u_{c}(t)\mathrm{d}t$ and $\overline{u}_{uc} \coloneqq \frac{1}{t_f}\int_{0}^{t_f}u_{uc}(t)\mathrm{d}t$ are the mean values of the controlled and uncontrolled inputs. As with the nominal case and similar to \cite{PO25a}, we assume $B_c$ has full row rank. {Note that the right-hand side equals $v$ as defined earlier. For a fixed $v$, the above equation is linear in $\overline{u}_c$, and similar to the nominal case, we have the least-norm solution
\begin{equation} \label{eq:ucLS}
\overline{u}_{c}^{LS} = \frac{1}{t_f}B_{c}^{\dagger}\left(v-\tilde{x}-t_fB_{uc}\overline{u}_{uc}\right),
\end{equation}
where $\|v\|_{\infty} \leq \overline{v}$ from \eqref{eq:v_bound}. We note here that \eqref{eq:uLS} and \eqref{eq:ucLS} provide a characterization of the controls ensuring $x_{tg}$ is reachable and resiliently reachable, thus addressing Problem~\ref{prob:char}. We also remind the reader that $v$ is not a fixed quantity, but the norm of $\overline{u}_{c}^{LS}$ can be used to approximate the malfunctioning energy $E_{M}^{*}(x_0, x_{tg}, t_f, u_{uc})$ in \eqref{eq:EM_def} as $v$ is bounded. First, we investigate conditions under which the constraint $u_c \in \Uc$ is satisfied using \eqref{eq:ucLS} and the bound $\overline{v}$.}

{If $\|u_c(t)\|_{\infty} \leq 1$, then $\overline{u}_{c}^{LS}$ in \eqref{eq:ucLS} satisfies $\|\overline{u}_{c}^{LS}\|_{\infty} \leq 1$. We also note that $\|u_{uc}(t)\|_{\infty} \leq 1$ in \eqref{eq:SetUuc} implies $\|\overline{u}_{uc}\|_{\infty} \leq 1$.} As with the nominal case, reducing the expression $\|\overline{u}_{c}^{LS}\|_{\infty} \leq 1$ to a condition on $t_f$ is not straightforward due to the presence of $v(t_f, u)$, as well as the nonlinear dependence of the expression on $t_f$. We can again check whether $\|\overline{u}_{c}^{LS}\|_{\infty} \leq 1$ for a given $t_f$ by using the known bound $\|v\|_{\infty} \leq \overline{v}$ in the following condition:
\begin{equation} \label{eq:uccond}
\max_{\|v\|_{\infty} \leq \overline{v}} \max_{\|\overline{u}_{uc}\|_{\infty} \leq 1} \frac{1}{t_f} \left\|B_{c}^{\dagger}\left(v-\tilde{x}-t_fB_{uc}\overline{u}_{uc}\right)\right\|_{\infty} \leq 1.
\end{equation}
As before, if \eqref{eq:uccond} is satisfied, then $\|\overline{u}_{c}^{LS}\|_{\infty} \leq 1$ is true for all $v(t_f, u)$ and all $u_{uc} \in \Uuc$, which is necessary to ensure $u_c \in \Uc$. Equation \eqref{eq:uccond} reduces to the following conditions:
\begin{align*}
&\max_{\|v\|_{\infty} \leq \overline{v}} \max_{\|\overline{u}_{uc}\|_{\infty} \leq 1} \frac{1}{t_f} \left[B_{c}^{\dagger}(v-\tilde{x}-t_fB_{uc}\overline{u}_{uc})\right]_{j} \leq 1 \\
\text{and } &\max_{\|v\|_{\infty} \leq \overline{v}} \max_{\|\overline{u}_{uc}\|_{\infty} \leq 1} -\frac{1}{t_f} \left[B_{c}^{\dagger}(v-\tilde{x}-t_fB_{uc}\overline{u}_{uc})\right]_{j} \leq 1
\end{align*}
for each $j = 1, \ldots, m$. The objective of the inner maximum is linear in $\overline{u}_{uc}$, and hence its solution lies on the vertices of the hypercube $\|\overline{u}_{uc}\|_{\infty} \leq 1$. Following this, the objective of the outer maximum is also linear in $v$, and the solution lies on the vertices of the hypercube $\|v\|_{\infty} \leq \overline{v}$. Thus, the above conditions can be checked efficiently on the vertices of the two hypercubes $\|\overline{u}_{uc}\|_{\infty} \leq 1$ and $\|v\|_{\infty} \leq \overline{v}$ \cite{LinOpt}.

{Finally, we consider the malfunctioning energy $E_{M}^{*}(x_0, x_{tg}, t_f, u_{uc})$ in \eqref{eq:EM_def} and note that obtaining an exact expression is intractable \cite{CVOC} for the general nonlinear system \eqref{eq:Malfunctioning}. We thus use \eqref{eq:ucLS} to approximate $E_{M}^{*}(x_0, x_{tg}, t_f, u_{uc})$ using Lemma \ref{lem:approx}. From \eqref{eq:approx}, we have $E_{M}^{*}(x_0, x_{tg}, t_f, u_{uc}) \geq t_f\|\overline{u}_{c}^{LS}\|_{2}^{2}$, where we choose the mean of the optimal control signal to satisfy \eqref{eq:ucLS}, motivated by \cite[Lemma 1]{PO25a}. Then, similar to the approximation for the nominal energy \eqref{eq:ENapprox}, we write the approximate malfunctioning energy 
\begin{equation} \label{eq:EMapprox}
E_{M}^{*}(x_0, x_{tg}, t_f, u_{uc}) \approx t_f\|\overline{u}_{c}^{LS}\|_{2}^{2},
\end{equation}
where we recall that $\overline{u}_{c}^{LS}$ depends on $v(t_f, u)$, which is bounded as in \eqref{eq:v_bound}. As before, this approximation is exact when the optimal control is constant, i.e., $u_{c}^{*}(t) = \overline{u}_{c}^{LS}$ for all $t \in [0, t_f]$, which is true in the case of linear driftless systems as shown in \cite{PO25a}. Under the assumption that this approximation is exact, we obtain
\begin{equation} \label{eq:EM}
E_{M}^{*}(x_0, x_{tg}, t_f, u_{uc}) = \frac{1}{t_f}\left\|B_{c}^{\dagger}\left(v-\tilde{x}-t_fB_{uc}\overline{u}_{uc}\right)\right\|_{2}^{2},
\end{equation}
with $\|v\|_{\infty} \leq \overline{v}$.}

{Equations \eqref{eq:EN} and \eqref{eq:EM} provide the approximations for the nominal and malfunctioning energies and are obtained using Proposition \ref{prop:v} and Lemma \ref{lem:approx}. In Appendix \ref{app:B}, we list all approximate and non-approximate quantities in this paper for clarity. The conditions \eqref{eq:ucond} and \eqref{eq:uccond} are used to check whether the constraints $u \in \U$ and $u_c \in \Uc$ are satisfied for a given $t_f$. In the next subsection, we use these results to obtain an approximate upper bound on the worst-case total energy $\overline{E}_{T}(x_0, x_{tg}, t_f)$ defined in \eqref{eq:ET_bar_def}.}

\subsection{Worst-case Total Energy}
We now consider the worst-case total energy $\overline{E}_{T}(x_0, x_{tg}, t_f)$, defined in \eqref{eq:ET_bar_def}. {Note that the expression satisfied by the mean of the controlled input $\overline{u}_{c}^{LS}$ in \eqref{eq:ucLS} and the approximate malfunctioning energy \eqref{eq:EM} depend on the uncontrolled input $u_{uc}$. The maximal effect of $u_{uc}$ is thus quantified by the worst-case total energy, representing the energy used by all actuators of the malfunctioning system. The following proposition provides an upper bound on the worst-case total energy based on the approximations \eqref{eq:EN} and \eqref{eq:EM}.}

\begin{proposition}[Worst-case Total Energy] \label{prop:ET_bar}
Under the assumption that the approximations for the nominal \eqref{eq:EN} and malfunctioning \eqref{eq:EM} energies are exact, the worst-case total energy $\overline{E}_{T}(x_0, x_{tg}, t_f)$ is bounded from above as follows:
\begin{equation} \label{eq:ET_bar}
\overline{E}_T(x_0, x_{tg}, t_f) \leq \frac{1}{t_f}\left\|B_{c}^{\dagger}(v-\tilde{x})\right\|_{2}^{2} + t_f\left(\sum_{i = 1}^{p} \lambda_i \|V\|_{1}^{2} + p\right) + 2\left\|B_{uc}^{T}B_{c}^{\dagger T}B_{c}^{\dagger}(v-\tilde{x})\right\|_{1},
\end{equation}
where $p$ is the dimension of the uncontrolled input $u_{uc}$, and $B_{uc}^{T}B_{uc}$ has the spectral decomposition $B_{uc}^{T}B_{uc} = V\Lambda V^T$, with $\Lambda = \diag\{\lambda_1, \ldots, \lambda_p\}$.
\end{proposition}
\begin{proof}
From the definition of the worst-case total energy \eqref{eq:ET_bar_def} and the expression for the malfunctioning energy \eqref{eq:EM},
$$
\overline{E}_T(x_0, x_{tg}, t_f) = \sup_{u_{uc} \in \Uuc} \frac{1}{t_f}\left\|B_{c}^{\dagger}\left(v-\tilde{x}-t_fB_{uc}\overline{u}_{uc}\right)\right\|_{2}^{2} + \int_{0}^{t_f} u_{uc}^{T}(t)u_{uc}(t)\mathrm{d}t.
$$
Obtaining a closed-form analytical expression for the above problem is not straightforward, and hence we aim to bound these terms from above. Splitting the supremum over three different terms, we obtain the following bound:
\begin{align}
\overline{E}_T(x_0, x_{tg}, t_f) &\leq \frac{1}{t_f}\left\|B_{c}^{\dagger}(v-\tilde{x})\right\|_{2}^{2} + \underbrace{\sup_{\|\overline{u}_{uc}\|_{\infty} \leq 1} t_f \overline{u}_{uc}^{T}B_{uc}^{T}B_{uc}\overline{u}_{uc}}_{=T_1} + \underbrace{\sup_{\|\overline{u}_{uc}\|_{\infty} \leq 1} 2(\tilde{x}-v)^{T}B_{c}^{\dagger T}B_{c}^{\dagger}B_{uc}\overline{u}_{uc}}_{=T_2} \nonumber \\
&+ \underbrace{\sup_{u_{uc} \in \Uuc} \int_{0}^{t_f} u_{uc}^{T}(t)u_{uc}(t)\mathrm{d}t}_{=T_3}, \label{eq:ETbar1}
\end{align}
where $u_{uc} \in \Uuc$ implies the constraint $\|\overline{u}_{uc}\|_{\infty} \leq 1$ holds in the terms $T_1$ and $T_2$. We consider each term in \eqref{eq:ETbar1} separately, beginning with $T_1$. Let the spectral decomposition of $B_{uc}^{T}B_{uc}$ be $B_{uc}^{T}B_{uc} = V\Lambda V^T$, where $\Lambda = \diag\{\lambda_1, \ldots, \lambda_p\}$ and $V$ is the orthonormal matrix of eigenvectors of $B_{uc}^{T}B_{uc}$. Next, define $q \coloneqq V^T\overline{u}_{uc}$. Using the sub-multiplicative property of norms, $\|q\|_{\infty} \leq \|V^T\|_{\infty} \|\overline{u}_{uc}\|_{\infty} \leq \|V\|_1$ since $\|\overline{u}_{uc}\|_{\infty} \leq 1$. Due to the use of the sub-multiplicative property, the set $\left\{\overline{u}_{uc} : \|\overline{u}_{uc}\|_{\infty} \leq 1\right\}$ is a subset of the set $\left\{q:\|q\|_{\infty} \leq \|V\|_1\right\}$, where $q = V^T\overline{u}_{uc}$. This property results in the following upper bound for $T_1$:
\begin{equation} \label{eq:T1}
T_1 \leq t_f \sup_{\|q\|_{\infty} \leq \|V\|_1} q^T\Lambda q = t_f\sup_{\|q\|_{\infty} \leq \|V\|_1} \sum_{i = 1}^{p} \lambda_iq_{i}^{2} = t_f\sum_{i = 1}^{p} \lambda_i \|V\|_{1}^{2}.
\end{equation}
In \eqref{eq:T1}, each component of the $q^*$ achieving the supremum takes on the maximum possible value, i.e., $q_{i}^{*} = \pm \|V\|_1$. Next, the term $T_2$ can be rewritten as the problem $\sup_{\|\overline{u}_{uc}\|_{\infty} \leq 1} \beta^T\overline{u}_{uc} = \|\beta\|_1$, where $\beta = 2B_{uc}^{T}B_{c}^{\dagger T}B_{c}^{\dagger}(\tilde{x}-v)$ and the optimal $\overline{u}_{uc}$ for this particular term is $\overline{u}_{uc}^{*} = \sign(\beta) = \sign\left(2B_{uc}^{T}B_{c}^{\dagger T}B_{c}^{\dagger}(\tilde{x}-v)\right)$. Then,
\begin{equation} \label{eq:T2}
T_2 = 2\left\|B_{uc}^{T}B_{c}^{\dagger T}B_{c}^{\dagger}(v-\tilde{x})\right\|_{1}.
\end{equation}
Finally, $T_3$ is clearly maximized by each component of the uncontrolled input $u_{uc}$ taking on the maximal value of $+1$ or $-1$ at each time instant $t$, so that $u_{uc}^{T}(t)u_{uc}(t) = p$ for all $t$, and 
\begin{equation} \label{eq:T3}
T_3 = t_fp.
\end{equation}
Using \eqref{eq:T1}, \eqref{eq:T2} and \eqref{eq:T3} in \eqref{eq:ETbar1}, the result \eqref{eq:ET_bar} follows. 
\end{proof}

Proposition \ref{prop:ET_bar} thus provides an upper bound on the worst-case total energy, and we emphasize that this upper bound assumes the approximations for the nominal \eqref{eq:EN} and malfunctioning \eqref{eq:EM} energies are exact. In the next subsection, we consider the special case when control authority is lost over one actuator, i.e., when $p = 1$. In this particular case, we derive a closed-form expression instead of an upper bound for $\overline{E}_T(x_0, x_{tg}, t_f)$, assuming the above approximations are still exact. We subsequently use this closed-form expression to obtain a bound on the energetic resilience metric \eqref{eq:rA_def} in Section \ref{sec:Resilience}.

\subsection{Loss of Authority over One Actuator} \label{sec:1Act}
When control authority is lost over one actuator, we have $p = 1$, $u_{uc}(t) \in \mathbb{R}$ and $B_{uc} \in \mathbb{R}^n$. Then, the worst-case total energy $\overline{E}_T(x_0, x_{tg}, t_f)$ is given by the following proposition.

\begin{proposition}
Under loss of control authority over one actuator, i.e., when $p = 1$, and under the assumption that the approximations for the nominal \eqref{eq:EN} and malfunctioning \eqref{eq:EM} energies are exact, the worst-case total energy $\overline{E}_T(x_0, x_{tg}, t_f)$ is expressed as 
\begin{equation} \label{eq:ET_bar_1Act}
\overline{E}_T(x_0, x_{tg}, t_f) = \frac{1}{t_f}\left\|B_{c}^{\dagger}(v-\tilde{x})\right\|_{2}^{2} + t_f\left(\|B_{uc}\|_{2}^{2}+1\right) + 2\left|B_{uc}^{T}B_{c}^{\dagger T}B_{c}^{\dagger}(v-\tilde{x})\right|.
\end{equation}
\end{proposition}
\begin{proof}
From the definition \eqref{eq:ET_bar_def} and the expression for the malfunctioning energy \eqref{eq:EM},
\begin{align}
\overline{E}_T(x_0, x_{tg}, t_f) = \frac{1}{t_f}\left\|B_{c}^{\dagger}(v-\tilde{x})\right\|_{2}^{2} + \sup_{u_{uc} \in \Uuc} \biggl\{&t_fB_{uc}^{T}B_{uc} \overline{u}_{uc}^{2} + 2(v-\tilde{x})^{T}B_{c}^{\dagger T}B_{c}^{\dagger}B_{uc}\overline{u}_{uc}
+ \int_{0}^{t_f} u_{uc}^{2}(t)\mathrm{d}t\biggr\} \label{eq:ET2}
\end{align}
The constraint $u_{uc} \in \Uuc$ is equivalent to the constraints $-1 \leq \overline{u}_{uc} \leq 1$ on $\overline{u}_{uc}$. Then, the second term inside the supremum reduces to
$$
\sup_{-1 \leq \overline{u}_{uc} \leq 1} 2(v-\tilde{x})^{T}B_{c}^{\dagger T}B_{c}^{\dagger}B_{uc}\overline{u}_{uc} = 2\left|B_{uc}^{T}B_{c}^{\dagger T}B_{c}^{\dagger}(v-\tilde{x})\right|,
$$%
with the optimal $\overline{u}_{uc}^{*}$ for this particular term given by $\overline{u}_{uc}^{*} = \sign\left(B_{uc}^{T}B_{c}^{\dagger T}B_{c}^{\dagger}(v-\tilde{x})\right)$. When the outcome of the $\sign$ function is nonzero, this choice of $\overline{u}_{uc}^{*} \in \{-1,+1\}$ also maximizes the first term inside the supremum in \eqref{eq:ET2}, with $\sup_{-1 \leq \overline{u}_{uc} \leq 1} t_fB_{uc}^{T}B_{uc} \overline{u}_{uc}^{2} = t_f \|B_{uc}\|_{2}^{2}$. Further, $\overline{u}_{uc}^{*} \in \{-1, +1\}$ implies that the uncontrolled input $u_{uc}^{*}(t)$ also takes on a constant value of $-1$ or $+1$ at each instant $t \in [0, t_f]$. This value for $u_{uc}^{*}$ maximizes the third term inside the supremum in \eqref{eq:ET2}, with $\sup_{u_{uc} \in \Uuc} \int_{0}^{t_f} u_{uc}^{2}(t)\mathrm{d}t = t_f$.

{On the other hand, if $B_{uc}^{T}B_{c}^{\dagger T}B_{c}^{\dagger}(v-\tilde{x}) = 0$, the second term inside the supremum in \eqref{eq:ET2} vanishes, leading to a smaller value for $\overline{E}_T(x_0, x_{tg}, t_f)$ than if this term did not vanish. The worst-case total energy is thus achieved only when $B_{uc}^{T}B_{c}^{\dagger T}B_{c}^{\dagger}(v-\tilde{x}) \neq 0$. The uncontrolled input $u_{uc}^{*}$ maximizing each of the three terms inside the supremum in \eqref{eq:ET2} is the same, and is a constant equal to $-1$ or $+1$ at almost every instant $t \in [0, t_f]$. Combining the results of each term above, we obtain the result \eqref{eq:ET_bar_1Act}.}
\end{proof}

\begin{remark}
We remind the reader that the primary results of this section rely on approximations made on the basis of Lemma \ref{lem:approx}. Appendix \ref{app:B} provides a classification of all non-approximate and approximate quantities derived in this paper. Further, the expressions for the different energies in \eqref{eq:EN}, \eqref{eq:EM} and \eqref{eq:ET_bar} depend on $v \equiv v(t_f, u)$, with $\|v\|_{\infty} \leq \overline{v}$. These expressions play a key role in determining bounds on the energetic resilience metric \eqref{eq:rA_def}, as shown in the next section. 
\end{remark}

\section{Resilience of Nonlinear Systems} \label{sec:Resilience}
{In this section, we quantify the maximal additional energy used by the malfunctioning system \eqref{eq:Malfunctioning}, i.e., addressing Problem \ref{prob:resilience}, for the particular case of loss of control authority over one actuator. Namely, we derive a bound on the energetic resilience metric \eqref{eq:rA_def}, assuming the approximation for the nominal \eqref{eq:EN} and worst-case total \eqref{eq:ET_bar_1Act} energies are exact. We also discuss the reduction to the special case of linear driftless dynamics in \eqref{eq:Nominal} and recover an existing result from \cite{PO25a} in this case. We note that the more general case where authority is lost over $p>1$ actuators can be addressed using \eqref{eq:ET_bar}. However, this bound is more conservative as \eqref{eq:ET_bar} provides only an upper bound on the approximate worst-case total energy.}

First, given $\tilde{x} = x_0 - x_{tg}$, we rewrite the definition \eqref{eq:rA_def} as
\begin{align*}
r_A(t_f, R) &= \sup_{\|\tilde{x}\|_2 \leq R} \overline{E}_T(x_0,x_{tg},t_f) - E_{N}^{*}(x_0, x_{tg}, t_f).
\end{align*}
Next, note that the nominal \eqref{eq:EN} and worst-case total \eqref{eq:ET_bar} energies depend only on $v-\tilde{x}$, and not directly on $\tilde{x}$ itself. Then, we can use the bounds $\|\tilde{x}\|_2 \leq R$ and $\|v\|_{\infty} \leq \overline{v}$ to obtain
$$
\|v-\tilde{x}\|_2 \leq \|v\|_2 + \|\tilde{x}\|_2 \leq \sqrt{n}\|v\|_{\infty} + \|\tilde{x}\|_2 \leq R + \overline{v}\sqrt{n},
$$%
where we note that the norm inequality $\|v\|_2 \leq \sqrt{n}\|v\|_{\infty}$ holds for $v \in \mathbb{R}^n$ \cite[Chapter 5]{Horn}. {We then replace the constraint $\|\tilde{x}\|_2 \leq R$ with $\|v-\tilde{x}\|_2 \leq R+\overline{v}\sqrt{n}$. Due to the norm inequality above, the set $\{d:\|\tilde{x}\|_2 \leq R\}$ is a subset of the set $\{d: \|v-\tilde{x}\|_2 \leq R + \overline{v}\sqrt{n} \text{ where } \|v\|_{\infty} \leq \overline{v}$\}. This property results in the following upper bound on $r_A(t_f, R)$:}
\begin{equation} \label{eq:rA_vd}
r_A(t_f, R) \leq \sup_{\|v-\tilde{x}\|_2 \leq R+\overline{v}\sqrt{n}} \overline{E}_T(x_0,x_{tg},t_f) - E_{N}^{*}(x_0, x_{tg}, t_f).
\end{equation}
We note that this analysis demonstrates why the metric used in \cite{PO25a} is unsuitable for nonlinear systems. Briefly, the metric in \cite{PO25a} uses a constraint of the form $\|\tilde{x}\|_2 \geq R$ rather than $\|\tilde{x}\|_2 \leq R$ in \eqref{eq:rA_def}. As before, since \eqref{eq:EN} and \eqref{eq:ET_bar} depend only on $v-\tilde{x}$, we would require a lower bound on $\|v-\tilde{x}\|_2$ to replace $\|\tilde{x}\|_2 \geq R$. This lower bound can be derived as $\|v-\tilde{x}\|_2 \geq R-\overline{v}\sqrt{n}$, but is meaningful only when $R > \overline{v}\sqrt{n}$. However, $\overline{v}$ in \eqref{eq:v_bound} can take on extremely large values, thus preventing $R > \overline{v}\sqrt{n}$ from being satisfied. Thus, we have introduced a different metric \eqref{eq:rA_def} in this paper for nonlinear systems.

We can then use \eqref{eq:rA_vd} to obtain the following bound on the resilience metric.

\begin{proposition} \label{prop:rA_gen}
The energetic resilience $r_A(t_f, R)$ is bounded from above as follows:
\begin{equation} \label{eq:rA_gen}
r_A(t_f, R) \leq t_f\Bigl(\sum_{i = 1}^{p} \lambda_i \|V\|_{1}^{2} + p\Bigr) + \frac{(R+\overline{v}\sqrt{n})^2}{t_f} \lambda_{\max}\left(B_{c}^{\dagger T}B_{c}^{\dagger} - B^{\dagger T}B^{\dagger}\right) + 2n(R+\overline{v}\sqrt{n}) \left\|B_{uc}^{T}B_{c}^{\dagger T}B_{c}^{\dagger}\right\|_2.
\end{equation}
\end{proposition}

\begin{proof}
Using \eqref{eq:EN} and \eqref{eq:ET_bar} in \eqref{eq:rA_vd},
\begin{align*}
r_A(t_f, R) &\leq \sup_{\|v-\tilde{x}\|_2 \leq R+\overline{v}\sqrt{n}} \overline{E}_T(x_0,x_{tg},t_f) - E_{N}^{*}(x_0, x_{tg}, t_f) \\
&\leq t_f\Bigl(\sum_{i = 1}^{p} \lambda_i \|V\|_{1}^{2} + p\Bigr) + \sup_{\|v-\tilde{x}\|_2 \leq R+\overline{v}\sqrt{n}} \frac{1}{t_f}(v-\tilde{x})^{T}\left(B_{c}^{\dagger T}B_{c}^{\dagger} - B^{\dagger T}B^{\dagger}\right)(v-\tilde{x}) \\
&\hspace{9.7em} + \sup_{\|v-\tilde{x}\|_2 \leq R+\overline{v}\sqrt{n}} 2\left\|B_{uc}^{T}B_{c}^{\dagger T}B_{c}^{\dagger}(v-\tilde{x})\right\|_{1} \\
&\leq t_f\Bigl(\sum_{i = 1}^{p} \lambda_i \|V\|_{1}^{2} + p\Bigr) + 2n(R+\overline{v}\sqrt{n}) \left\|B_{uc}^{T}B_{c}^{\dagger T}B_{c}^{\dagger}\right\|_2 \\
&\hspace{9.7em}+ \sup_{\|v-\tilde{x}\|_2 \leq R+\overline{v}\sqrt{n}} \frac{1}{t_f}(v-\tilde{x})^{T}\left(B_{c}^{\dagger T}B_{c}^{\dagger} - B^{\dagger T}B^{\dagger}\right)(v-\tilde{x}),
\end{align*}
where the second term above follows from the norm inequality $\|z\|_1 \leq \sqrt{n}\|z\|_2$ for any vector $z \in \mathbb{R}^n$. Since $B_{c}^{\dagger T}B_{c}^{\dagger} - B^{\dagger T}B^{\dagger}$ is symmetric, we can use the Rayleigh inequality \cite{Horn} and the constraint $\|v-\tilde{x}\|_2 \leq R+\overline{v}\sqrt{n}$ to bound the third term as
$$
\sup_{\|v-\tilde{x}\|_2 \leq R+\overline{v}\sqrt{n}} \frac{1}{t_f}(v-\tilde{x})^{T}\left(B_{c}^{\dagger T}B_{c}^{\dagger} - B^{\dagger T}B^{\dagger}\right)(v-\tilde{x}) \leq \frac{(R+\overline{v}\sqrt{n})^2}{t_f} \lambda_{\max}\left(B_{c}^{\dagger T}B_{c}^{\dagger} - B^{\dagger T}B^{\dagger}\right),
$$
from which the result \eqref{eq:rA_gen} follows.
\end{proof}

We also specialize this result to the case when $p = 1$ in the following corollary.

\begin{corollary} \label{prop:rA}
When $p = 1$, the energetic resilience $r_A(t_f, R)$ is bounded from above as follows:
\begin{align}
r_A(t_f, R) \leq \frac{(R+\overline{v}\sqrt{n})^2}{t_f} \lambda_{\max}\left(B_{c}^{\dagger T}B_{c}^{\dagger} - B^{\dagger T}B^{\dagger}\right) &+ 2(R+\overline{v}\sqrt{n}) \left\|B_{uc}^{T}B_{c}^{\dagger T}B_{c}^{\dagger}\right\|_2 + t_f\left(\|B_{uc}\|_{2}^{2} + 1\right). \label{eq:rA}
\end{align}
\end{corollary}

\begin{remark}[Reduction to Linear Driftless Systems] \label{rem:Dft}
{We recall that setting $f(x(t)) \equiv 0$ and $g(x(t)) \equiv B$ recovers the driftless system \eqref{eq:Nominal_Dft} from \eqref{eq:Nominal}. As discussed in Section \ref{sec:Energies}, the quantity $v(t_f, u)$ reduces to $0$ so that $\overline{v} = 0$. In this special case, equality is achieved in \eqref{eq:ENapprox}, so that \eqref{eq:EN} and \eqref{eq:EM} are no longer approximations, but hold with equality. Further, we can easily observe that the energies \eqref{eq:EN}, \eqref{eq:EM}, \eqref{eq:ET_bar} and \eqref{eq:ET_bar_1Act} reduce to the corresponding expressions \eqref{eq:EN_Dft}, \eqref{eq:EM_Dft}, \eqref{eq:ETbar_Dft} and \eqref{eq:ETbar_Dft_1Act} for the linear driftless system. The bound \eqref{eq:rA} on the energetic resilience metric also reduces to the corresponding bound \eqref{eq:rA_bound} for linear driftless dynamics. We emphasize that each of these expressions for driftless dynamics hold without the need for approximations.}
\end{remark}


Finally, it is useful to investigate whether the bounds on the energetic resilience are tight. We now demonstrate that these bounds are achieved for the class of linear driftless systems with one state, losing authority over one input.

\begin{proposition} \label{prop:achieve}
The bound \eqref{eq:rA_gen} on the energetic resilience is exactly achieved for linear driftless systems of the form $\dot{x}(t) = b_cu_c(t) + b_{uc}u_{uc}(t)$, where all quantities are scalars.
\end{proposition}
\begin{proof}
Note that in the special case of linear driftless systems, as per Remark~\ref{rem:Dft}, we recover all expressions derived in Section~\ref{sec:Res_Dft}. Further, we know that $\overline{v} = 0$ in this case. Let $B = [b_c ~~ b_{uc}] \in \mathbb{R}^{1\times 2}$. Then,
$$
E_{N}^{*}(x_0, x_{tg}, t_f) = \frac{1}{t_f} \|B^{\dagger}\tilde{x}\|_{2}^{2} = \frac{\tilde{x}^2}{t_f}\|B^{\dagger}\|_{2}^{2}
$$
and
$$
\overline{E}_T(x_0, x_{tg}, t_f) = \frac{\tilde{x}^{2}}{t_fb_{c}^{2}} + t_f(1+b_{uc}^2) + 2\left|\frac{b_{uc}\tilde{x}}{b_{c}^{2}}\right|.
$$%
We note that these expressions are \emph{exact} and not approximate, since the dynamics are linear and driftless. We also note that this value of $\overline{E}_T(x_0, x_{tg}, t_f)$ is achieved by the constant uncontrolled input
$$
u_{uc}^*(t) = \sign\left(\frac{b_{uc}\tilde{x}}{b_{c}^{2}}\right)
$$
as shown in \cite{PO25a}. Substituting these into the definition of the metric,
\begin{align*}
r_A(t_f, R) &= t_f(1+b_{uc}^{2}) + \sup_{|\tilde{x}|\leq R} \frac{\tilde{x}^2}{t_f}\left(\|B^{\dagger}\|_{2}^{2} - \frac{1}{b_{c}^{2}}\right) + 2\left|\frac{b_{uc}\tilde{x}}{b_{c}^{2}}\right| \\
&= t_f(1+b_{uc}^{2}) + \frac{R^2}{t_f}\left(\|B^{\dagger}\|_{2}^{2} - \frac{1}{b_{c}^{2}}\right) + 2R\left|\frac{b_{uc}}{b_{c}^{2}}\right|.
\end{align*}
It is clearly seen that this is simply the upper bound (34) with $\overline{v} = 0$ since the system has linear driftless dynamics. The constant uncontrolled input above thus exactly achieves the upper bound in this case.
\end{proof}

{The bound in \eqref{eq:rA} thus quantifies the maximal additional energy needed for a nonlinear system which loses authority over one actuator, where we recall the approximations for the nominal \eqref{eq:EN} and worst-case total \eqref{eq:ET_bar_1Act} energies. We also remind the reader that Appendix \ref{app:B} provides a classification of non-approximate and approximate quantities derived in this paper. While the approximations may introduce an error in quantifying the actual energy used, we note that obtaining exact expressions for the energies \eqref{eq:EN_def} and \eqref{eq:EM_def}, and subsequently the metric \eqref{eq:rA_def} is difficult due to the nonlinear nature of the system dynamics. We also note that a potential limitation of this metric is that the bound \eqref{eq:rA} may be conservative, and overestimate the maximal energy required. This may lead to conservative capacity design in practical systems, where resources such as fuel economy are important considerations. However, in the following section, we present three examples demonstrating the use of the energetic resilience metric in characterizing the increased energy required without significant conservatism, in the case of driftless, linear and nonlinear dynamical systems.

\section{Simulation Examples} \label{sec:Examples}

We now present three examples demonstrating the accuracy of the bounds \eqref{eq:rA_bound} and \eqref{eq:rA} on the energetic resilience. In all examples, we show that these bounds are useful in characterizing the maximal additional energy used by a system \eqref{eq:Malfunctioning} to achieve a target state, while compensating for the effect of uncontrolled inputs $u_{uc}$.

\subsection{Example 1: Linear Driftless Model} Our first example is based on the driftless model of an underwater robot, also considered in \cite{PO25a}:
\begin{equation} \label{eq:Model}
\begin{bmatrix} \dot{x} \\ \dot{y} \end{bmatrix} = \begin{bmatrix} 2 & \phantom{-}1 & \phantom{-}1 \\ 0.2 & -1 & \phantom{-}1 \end{bmatrix} \begin{bmatrix} u_1 \\ u_2 \\ u_3 \end{bmatrix},
\end{equation}
where the input matrix is $B$. Here, $x$ and $y$ are the coordinates of the robot. The engine $u_1$ acts primarily along the $x$-direction, with a small bias in the $y$-direction. $u_2$ and $u_3$ are engines at a $45^{\circ}$ angle to the $x$- and $y$-directions, opposing each other. {The model in \cite{BO20} heavily biased the engine $u_1$ towards the $x$-direction, and the corresponding entry in the input matrix was $10$. We modify the entry to $2$, reducing the bias in favor of a more equitable distribution of the effect of each input.} 

\begin{figure}[!t]
	\centering
	\includegraphics[width = 0.55\textwidth]{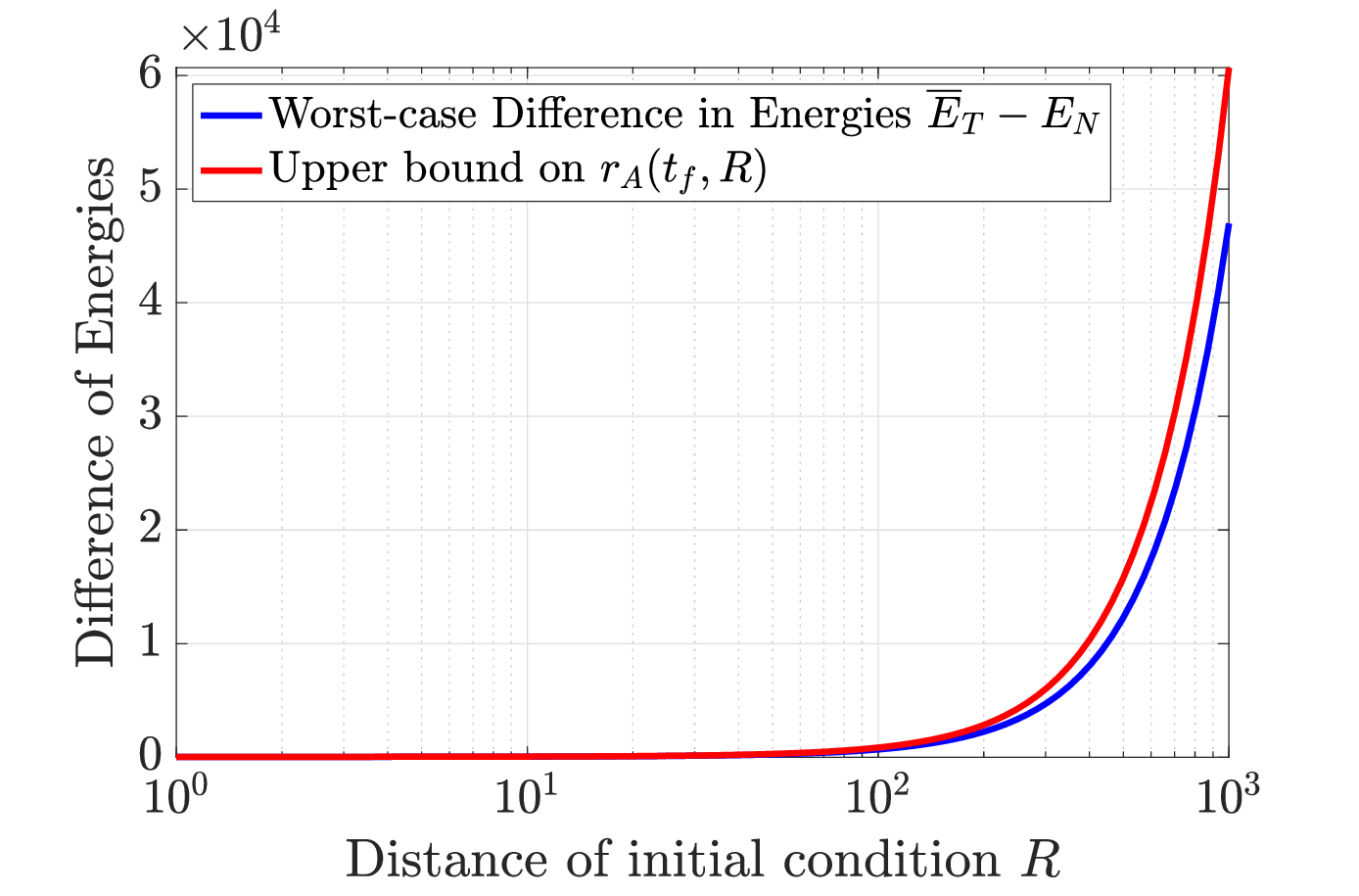}
	\caption{Illustration of energetic resilience metric as a function of distance of the initial condition from the origin for the linear driftless model of an underwater robot. The actual difference in energies $E_T-E_N$ does not vary significantly with different choices of uncontrolled inputs $u_{uc}$, and is hence not shown here.}
	\label{fig:Metric}
	\vspace{-0.4cm}
\end{figure}

We assume that a malfunction causes the system to lose authority over $u_3$ so that $B_c = \begin{bmatrix} 2 & \phantom{-}1 \\ 0.2 & -1 \end{bmatrix}$ and $B_{uc} = \begin{bmatrix} 1 \\ 1 \end{bmatrix}.$ In what follows, for notational simplicity, we denote $E_N \equiv E_{N}^{*}(x_0, x_{tg}, t_f)$, $\overline{E}_T \equiv \overline{E}_T(x_0, x_{tg}, t_f)$ and $E_T \equiv E_{T}(x_0, x_{tg}, t_f, u_{uc})$. We now compare the nominal energy $E_N$ \eqref{eq:EN_Dft} to the total energy $E_T$ defined in \eqref{eq:ET_def} and worst-case total energy $\overline{E}_M$ in \eqref{eq:ETbar_Dft_1Act}. In particular, we demonstrate the accuracy of the energetic resilience bound \eqref{eq:rA_bound} in quantifying this comparison.



We fix $t_f = 10$ and vary the distance $R$ of the initial condition $x_0$ from the origin. Figure~\ref{fig:Metric} shows the difference between the worst-case total energy $\overline{E}_{T}$ and the nominal energy $E_N$ as a function of the distance $R$ of the initial condition $x_0$ from the origin. We also plot the upper bound on $r_A(t_f, R)$ from \eqref{eq:rA_bound} as a function of $R$, and note that this is a {useful} upper bound characterizing the additional energy used by the malfunctioning system. In particular, the error between the difference $\overline{E}_T-E_N$ and the upper bound on $r_A(t_f, R)$ is very small. For instance, when $R = 10^2$, the relative error between these quantities is around $20\%$. Even as $R$ increases, the relative error is consistently under $35\%$ for the entire range of $R$ considered. This example demonstrates how $r_A(t_f, R)$ can be used to characterize the maximal additional energy required in a driftless system as considered in \eqref{eq:Model}.

{We remark that the model \eqref{eq:Model} is a minor modification of the model considered in \cite{BO20}. Since that model heavily biases the effect of $u_1$ on $x$, this metric may be relatively less useful in that case in comparison to the model in \eqref{eq:Model}. The heavier weight contributes to poor conditioning of the matrices $B^{\dagger T}B^{\dagger}$ and $B_{c}^{\dagger T} B_{c}^{\dagger}$, resulting in a more conservative bound in \eqref{eq:rA_bound} than in the example presented here.}

\subsection{Example 2: Linear Model}
In this example, we consider the ADMIRE fighter jet model \cite{SDRA}, a popular benchmark for control applications \cite{KWT10, BO22b}. The linearized dynamics $\dot{x} = Ax + Bu$ for the subsystem associated with control actions were established in \cite{SDRA}, where
\begin{equation} \label{eq:ADMIRE_linear}
A = \begin{bmatrix} -0.9967  & 0 & 0.6176 \\ 0 &-0.5057 &0 \\ -0.0939 &0 &-0.2127 \end{bmatrix} \text{ and }~
B = \begin{bmatrix}
    0 & -4.2423 & 4.2423 & 1.4871 \\ 
    1.6532 & -1.2735 & -1.2735 & 0.0024 \\
    0 & -0.2805 & 0.2805 & -0.8823
\end{bmatrix}.
\end{equation}
The state $x = [p,q,r]^T$ consists of the roll, pitch and yaw rates measured in rad/s. The control input $u = [u_c, u_{re}, u_{le}, u_r]^T$ consists of deflections of the canard wings, right and left elevons and the rudder, measured in radians. Note that these inputs correspond to deflections from their equilibrium values in the linearization. We illustrate the use of the energetic resilience metric by calculating approximate control energies based only on the deflections. 

{Under the dynamics $\dot{x} = Ax + Bu$, where $A$ and $B$ are given in \eqref{eq:ADMIRE_linear}, we have $D_f = \|A\|_{\infty} = 1.6143$ and $D_g = 0$. We consider a loss of control authority over the canard wing so that $B_{uc}$ consists of the first column of $B$ in \eqref{eq:ADMIRE_linear}. It was shown in \cite{BO22b} that the system is resilient to the loss of this control input. We fix $t_f = 1$ and vary the distance $R$ of the initial condition $x_0$ from the target state $x_{tg}$, chosen to be $[0~0~0]^T$. We remark that while we do not directly check whether achieving this target state is always possible, tools for continuous reachability analysis \cite{CORA} can be used to verify this.} 

\begin{figure}[!t]
	\centering
	\includegraphics[width = 0.55\textwidth]{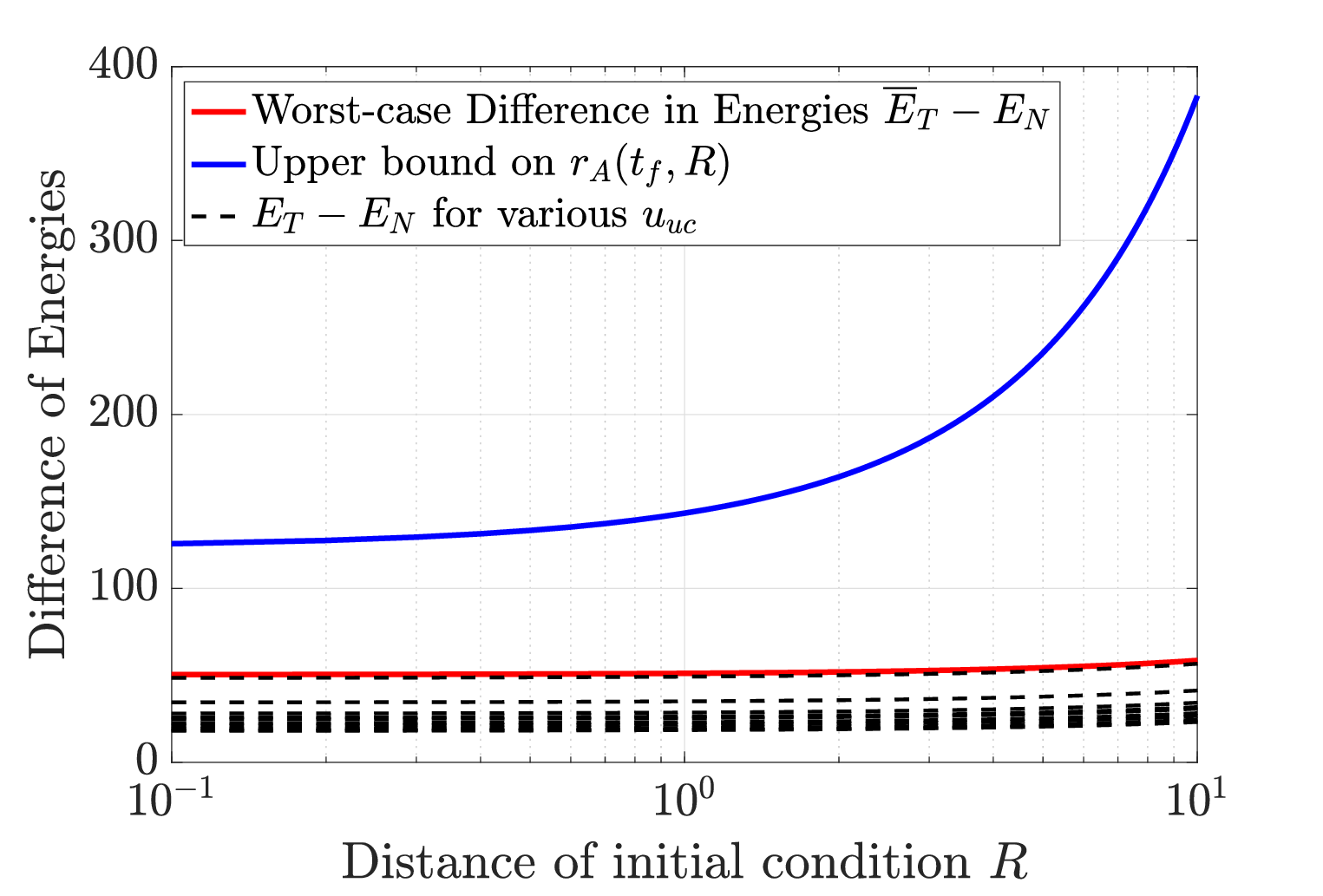}
	\caption{Difference between nominal and malfunctioning energies, compared to the upper bound on the metric $r_A(t_f, R)$, for a linear model of a fighter jet. The thinner dashed lines plot $E_T-E_N$ for a variety of uncontrolled inputs $u_{uc}$.}
	\label{fig:LinRes}
	\vspace{-0.4cm}
\end{figure}

{Fig.~\ref{fig:LinRes} compares the upper bound on $r_A(t_f,R)$ in \eqref{eq:rA} under the linear model for this example with the worst-case difference in energies $\overline{E}_T-E_N$ computed using the approximations \eqref{eq:EN} and \eqref{eq:ET_bar}. These quantities are plotted as a function of $R$, with $\|\tilde{x}\|_2 \leq R$. The thinner dashed lines in this figure show the difference between the approximate total and nominal energies $E_T-E_N$ for various uncontrolled inputs $u_{uc}$, computed using \eqref{eq:EN}, \eqref{eq:EM} and the definition \eqref{eq:ET_def}. We test uncontrolled inputs ranging sinusoids of varying frequencies, constants of different amplitudes and exponential decays.}

{It is clear that \eqref{eq:rA} bounds the difference in energies $E_T-E_N$, even under the worst-case uncontrolled input where we have the difference $\overline{E}_T-E_N$. As $R$ increases, we note that the bound on $r_A(t_f,R)$ increases faster than the difference in energies, presumably because $\overline{E}_T$, $E_T$ and $E_N$ all increase with $R$. This behavior contributes to smaller errors between the worst-case difference $\overline{E}_T-E_N$ and the bound on $r_A(t_f,R)$ when $R$ is smaller. We further note that these quantities are of the same order of magnitude for the range of $R$ considered, even when the relative error between them is large. 
Fig.~\ref{fig:LinRes} is thus an example of how the metric $r_A(t_f,R)$ can be used to estimate the maximal additional energy required to achieve a target state for a linear system.}

\subsection{Example 3: Wind-induced Nonlinear Model}
In this example, we consider the dynamics of the ADMIRE fighter jet model subjected to wind effects. The dynamic model can then be written as $\dot{x} = Ax + Bu + f_w(x)$, where $A$ and $B$ are given in \eqref{eq:ADMIRE_linear} and $f_w(x)$ is a nonlinear function of the state, modeling the effect of wind on the fighter jet model. We assume $f_w(x)$ consists of sinusoidal and constant expressions in each component as modeled in \cite{WWS24}, and that its Lipschitz constant is no more than $1$, so that $D_f = \|A\|_{\infty} + 1 = 2.6143$. In this particular instance, we choose $f_w(x) = \frac{1}{2}\left[\sin(p)\cos^2(p), ~-\sin(2q), ~1\right]^T$. {We note that when $f_w(x)$ is unknown, we can test the applicability of the metric by searching over functions with Lipschitz constants smaller than $1$.} As the dynamics are still linear in the input, we have $D_g = 0$. As before, we consider a loss of control authority over the canard.

\begin{figure}[!t]
	\centering
	\includegraphics[width = 0.55\textwidth]{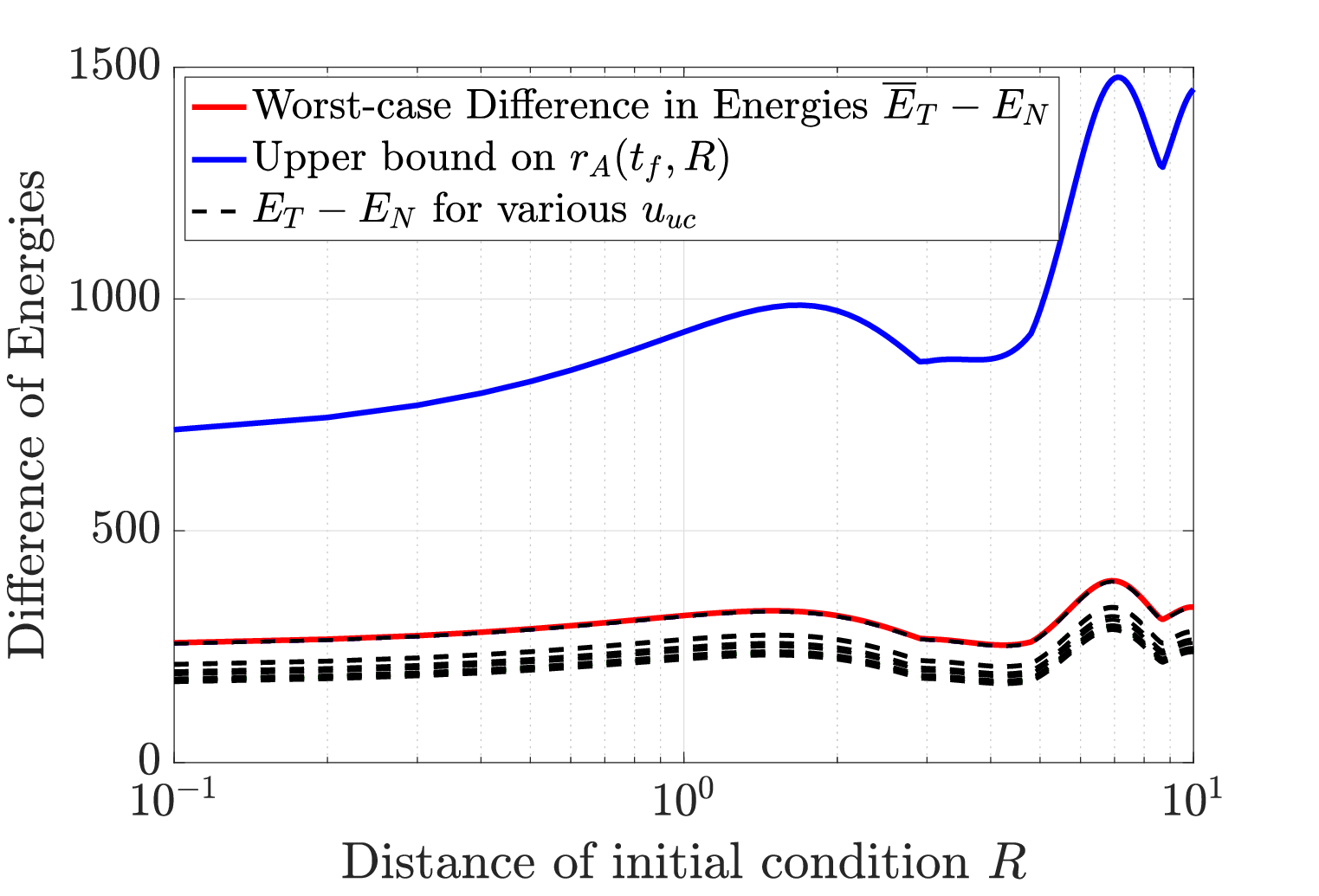}
	\caption{Difference between nominal and malfunctioning energies, compared to the upper bound on the metric $r_A(t_f, R)$, for a nonlinear wind-affected model. The thinner dashed lines plot $E_T-E_N$ for a variety of uncontrolled inputs $u_{uc}$.}
	\label{fig:NLRes}
	\vspace{-0.4cm}
\end{figure}

In Fig.~\ref{fig:NLRes}, we compare the upper bound on $r_A(t_f, R)$ in \eqref{eq:rA} to the worst-case difference in energies $\overline{E}_T - E_N$, where we use the approximations \eqref{eq:EN} and \eqref{eq:ET_bar} to compute these quantities. We plot these quantities as a function of the distance $R$ of the initial condition $x_0$ from the target state $x_{tg} = [0~0~0]^T$, and fix the final time as $t_f = 1$. As before, tools for reachability analysis can be used to verify whether the target state can be achieved under these settings. The thinner dashed lines denote the difference in energies $E_T - E_N$ for different uncontrolled inputs $u_{uc}$, computed using the approximations \eqref{eq:EN}, \eqref{eq:EM} and the definition \eqref{eq:ET_def}. As before, the uncontrolled inputs we test consist of sinusoids of varying frequencies, exponential decays and constant inputs of different amplitudes. It is clear that even under the worst-case uncontrolled input, the difference in energies $\overline{E}_T-E_N$ is bounded by \eqref{eq:rA}. As expected, the actual differences in energies $E_T-E_N$ for different uncontrolled inputs is also bounded by \eqref{eq:rA}.

{While the relative error between the worst-case difference $\overline{E}_T-E_N$ and the bound on $r_A(t_f,R)$ is large, we note that these quantities are of the same order of magnitude. This error is also smaller when $x_0$ is closer to $x_{tg}$, i.e., when $R$ is small. In this case, the bound on $r_A(t_f,R)$ is less than $3$ times larger than the worst-case difference in energies. We also note that the results for the nonlinear model in Fig.~\ref{fig:NLRes} show an improvement over the results for the linear model in Fig.~\ref{fig:LinRes}. This is due to the nonlinear function $f_w(x)$ that models wind effects. The energy required in the control signal can decrease when this function acts in the same direction that the system seeks to move towards. 

\begin{figure}[!t]
    \centering
    \includegraphics[width=0.5\textwidth]{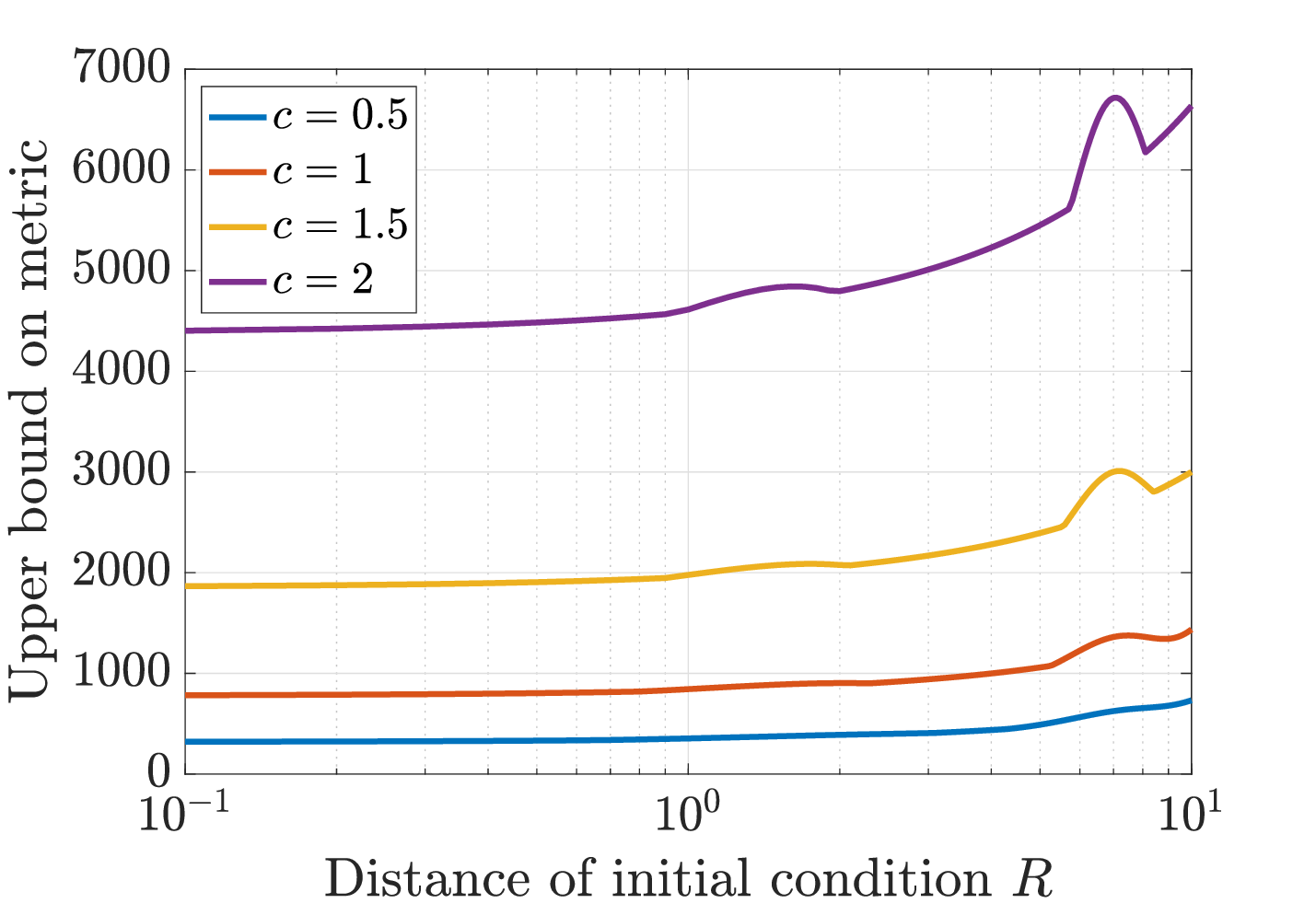}
    \caption{Sensitivity of metric to different amplitudes of wind disturbance.}
    \label{fig:parsen}
\end{figure}

In Fig.~\ref{fig:parsen}, we illustrate the sensitivity of the metric to different amplitudes of wind disturbances. In particular, we test disturbances of the form $f_w(x) = \frac{c}{2}\left[\sin(p)\cos^2(p), ~-\sin(2q), ~1\right]^T$, where $c$ is a varying parameter signifying the amplitude of disturbances. Note that $c = 1$ is the case considered in Fig.~\ref{fig:NLRes}. Clearly, the upper bound on the metric is sensitive to the choice of $c$, and for larger values of $c$, may exhibit more conservative behavior.

We remind the reader that the expressions plotted in Figs.~\ref{fig:LinRes} and \ref{fig:NLRes} are based on approximations made in \eqref{eq:EN} and \eqref{eq:EM}, resulting in approximations for $\overline{E}_T$ \eqref{eq:ET_bar} and $r_A(t_f,R)$ \eqref{eq:rA}. Under these approximations, we note that $r_A(t_f,R)$ is a useful quantity that estimates the increased energy required without significant conservatism, despite our knowledge of only its approximate upper bound.}

\section{Conclusions} \label{sec:Conclusions}
In this article, we introduce a metric to bound the maximal additional energy used by a nonlinear dynamical system in achieving a fixed-time reachability task, when it loses authority over a subset of its actuators. This maximal energy is compared to the nominal energy expended when no authority is lost over the inputs. {We first consider the special case of linear driftless systems and recall the energies used by the control signal in both the nominal and malfunctioning cases. This analysis is used to derive a bound on the metric to quantify the maximal additional energy used in the malfunctioning case, assuming control authority is lost over one actuator. For the general nonlinear case, we first derive an expression satisfied by the mean of the control signal used in the nominal and malfunctioning case. We then show that the energy in this mean value can be used to approximate the energy in the control signal in both cases. Under this approximation, we derive a bound on the worst-case malfunctioning energy over all possible malfunctioning inputs. In the case when authority is lost over only one actuator, this bound reduces to an equality under the aforementioned approximation. This expression is finally used to derive a bound on the resilience metric when authority is lost over one actuator. A set of simulation examples on driftless, linear and nonlinear dynamical systems demonstrate the applicability of this metric in characterizing the increased energy required, despite the approximations used for the control energies in the nonlinear case.}

We note that we have not numerically simulated optimal controllers for nonlinear systems, and thus do not compare the metric with the difference between the true optimal controllers for the nominal and malfunctioning systems. Such a numerical simulation is an avenue to explore in order to further demonstrate the usefulness of this metric. Another important avenue to explore is the actual design of resilient controllers. Such a resilient design must also consider that controlled and uncontrolled components may not be known in advance, and must adapt to changes in dynamics. A potential approach towards addressing this problem is presented in \cite{PAOO25} under a switched systems framework. Combining that approach with the quantification of additional energy in this paper is a useful avenue to explore. We also note that it is useful to study the case of losing authority over multiple actuators, through simulation or theory, though we may only obtain a more conservative bound in \eqref{eq:rA}.


\appendix
\section{Proof of Proposition \ref{prop:v}} \label{app:A}
Recall that
\begin{align*}
v &= -\int_{0}^{t_f} f(x(t))\mathrm{d}t - \int_{0}^{t_f} d_g(x(t),x_0)u(t)\mathrm{d}t = -t_ff_0 - \int_{0}^{t_f}d_f(x(t),x_0)\mathrm{d}t - \int_{0}^{t_f} d_g(x(t),x_0)u(t)\mathrm{d}t
\end{align*}
using \eqref{eq:f_cond}. Then,
$$
\|v\|_{\infty} \leq t_f \|f_0\|_{\infty} + \left\|\int_{0}^{t_f}d_f(x(t),x_0)\mathrm{d}t\right\|_{\infty} + \left\|\int_{0}^{t_f} d_g(x(t),x_0)u(t)\mathrm{d}t\right\|_{\infty}.
$$
Using Jensen's inequality to move the norm inside the integral, and subsequently using conditions \eqref{eq:f_cond} and \eqref{eq:g_cond}, we note that
\begin{align*}
\left\|\int_{0}^{t_f} d_g(x(t),x_0)u(t)\mathrm{d}t\right\|_{\infty} &\leq \int_{0}^{t_f} \left\|d_g(x(t),x_0)u(t)\right\|_{\infty} \mathrm{d}t \\
&\leq \int_{0}^{t_f} \|d_g(x(t),x_0)\|_{\infty} \mathrm{d}t \leq D_g \int_{0}^{t_f} \|x(t) - x_0\|_{\infty} \mathrm{d}t, \\
\text{and } \left\|\int_{0}^{t_f}d_f(x(t),x_0)\mathrm{d}t\right\|_{\infty} &\leq \int_{0}^{t_f} \|d_f(x(t),x_0)\|_{\infty} \mathrm{d}t \leq D_f \int_{0}^{t_f} \|x(t) - x_0\|_{\infty} \mathrm{d}t,
\end{align*}
where we use the sub-multiplicative property of matrix norms and the input constraint $\|u(t)\|_{\infty} \leq 1$ for all $t$ \eqref{eq:SetU}. Thus,
\begin{equation} \label{eq:v1}
\|v\|_{\infty} \leq t_f\|f_0\|_{\infty} + (D_f + D_g) \int_{0}^{t_f} \|x(t) - x_0\|_{\infty}\mathrm{d}t.
\end{equation}
We require a method to evaluate the second term in \eqref{eq:v1}. From \eqref{eq:Sol1},
$$
x(t)-x_0 = tf_0 + \int_{0}^{t} d_f(x(\tau),x_0)\mathrm{d}\tau + tB\overline{u} + \int_{0}^{t} d_g(x(\tau),x_0)u(\tau)\mathrm{d}\tau.
$$
Applying $\|\cdot\|_{\infty}$ on both sides and following steps similar to those used in obtaining \eqref{eq:v1}, we obtain
\begin{equation} \label{eq:x1}
\|x(t) - x_0\|_{\infty} \leq t(\|f_0\|_{\infty} + \|B\|_{\infty}) + (D_f + D_g)\int_{0}^{t} \|x(\tau) - x_0\|_{\infty}\mathrm{d}\tau.
\end{equation}
Define $Q(t) \coloneqq \int_{0}^{t}\|x(\tau) - x_0\|_{\infty}\mathrm{d}\tau$ and $c \coloneqq \|f_0\|_{\infty} + \|B\|_{\infty}$. Note that $Q(0) = 0$. In \eqref{eq:v1}, we require $Q(t_f)$ in the second term. From the first fundamental theorem of calculus \cite[Theorem 5.1]{Apostol}, $\dot{Q}(t) = \|x(t) - x_0\|_{\infty}$. Then, \eqref{eq:x1} reduces to
\begin{equation} \label{eq:Q1}
\dot{Q}(t) \leq ct + (D_f + D_g)Q(t), ~Q(0) = 0.
\end{equation}
Let $Q_1(t) \coloneqq ct + (D_f + D_g)Q(t), Q_1(0) = 0$ so that $\dot{Q}(t) \leq Q_1(t)$. Then, $\dot{Q}_1(t) = c + (D_f + D_g)\dot{Q}(t) \leq c + (D_f + D_g) Q_1(t)$. Next, let $Q_2(t) \coloneqq c + (D_f + D_g) Q_1(t), Q_2(0) = c$ so that $\dot{Q}_1(t) \leq Q_2(t)$. Then, $\dot{Q}_2(t) = (D_f + D_g)\dot{Q}_1(t) \leq (D_f + D_g)Q_2(t)$. We thus have the differential inequality
$$
\dot{Q}_2(t) \leq (D_f + D_g)Q_2(t), ~Q_2(0) = c.
$$
We now invoke Gr\"onwall's inequality \cite[Theorem 1.2.1]{BGP}, using which we obtain
$$
Q_2(t) \leq Q_2(0)e^{(D_f + D_g)t} = ce^{(D_f + D_g)t}.
$$
Substituting back in $Q_1(t)$ and $Q(t)$, we have
\begin{align*}
Q_1(t) &\leq \frac{c(e^{(D_f + D_g)t}-1)}{D_f + D_g} \\
\text{and } Q(t) &\leq \frac{c(e^{(D_f+D_g)t}-1-(D_f+D_g)t)}{(D_f + D_g)^2}.
\end{align*}
Thus,
\begin{equation} \label{eq:Qtf}
Q(t_f) \leq \frac{c\left(e^{t_f(D_f+D_g)}-1-t_f(D_f+D_g)\right)}{(D_f+D_g)^{2}}.
\end{equation}
Substituting the bound \eqref{eq:Qtf} on $Q(t_f) = \int_{0}^{t_f}\|x(t) - x_0\|_{\infty}\mathrm{d}t$ in the third term of \eqref{eq:v1},
$$
\|v\|_{\infty} \leq t_f\|f_0\|_{\infty} + \frac{c}{D_f+D_g} \left(e^{t_f(D_f+D_g)}-1-t_f(D_f+D_g)\right).
$$
Substituting $c = \|f_0\|_{\infty} + \|B\|_{\infty}$ and simplifying,
$$
\|v\|_{\infty} \leq \overline{v} \coloneqq \frac{1}{D_f+D_g} \left[c\left(e^{t_f(D_f + D_g)}-1\right)-t_f(D_f+D_g)\|B\|_{\infty}\right],
$$
thus concluding the proof. \hfill \qedsymbol

\section{Table of Non-Approximate and Approximate Expressions} \label{app:B}
Table \ref{tab:1} provides a classification of which equations are exact, i.e., non-approximate expressions or upper bounds, and which equations are based on the approximation from Lemma \ref{lem:approx}. We note that all expressions for linear driftless systems are non-approximate, and all expressions for nonlinear systems are approximate.

\begin{table}[!h]
\centering

\caption{Classification of Non-Approximate and Approximate Expressions}
\vspace{0.5em}
\begin{tabular}{c|c|c}
Dynamic Model & Classification & Equations \\ \toprule
\multirow{2}{*}{Linear Driftless System} & Non-Approximate Equalities & \eqref{eq:EN_Dft}, \eqref{eq:EM_Dft}, \eqref{eq:ETbar_Dft_1Act} \\ \cmidrule{2-3}
& Non-Approximate Upper Bounds & \eqref{eq:ETbar_Dft}, \eqref{eq:rA_bound} \\ \midrule
\multirow{2}{*}{Nonlinear System} & Approximate Equalities & \eqref{eq:EN}, \eqref{eq:EM}, \eqref{eq:ET_bar_1Act} \\ \cmidrule{2-3}
& Approximate Upper Bounds & \eqref{eq:ET_bar}, \eqref{eq:rA} \\ \bottomrule
\end{tabular}

\label{tab:1}

\end{table}

{ 
\bibliographystyle{IEEEtran}
\bibliography{references}
}

\end{document}